\documentclass[11pt, a4paper, twoside]{amsart}
\usepackage[centering, totalwidth = 380pt, totalheight = 590pt]{geometry}
\usepackage{amssymb, amsmath, amsthm, enumerate, microtype, stmaryrd, url}
\usepackage[latin1]{inputenc}
\usepackage[dvips, arrow, matrix, tips, curve]{xy}
\usepackage{color}
\definecolor{darkgreen}{rgb}{0,0.45,0}
\usepackage[pagebackref,colorlinks,citecolor=darkgreen,linkcolor=darkgreen]{hyperref}
\SelectTips{cm}{10}

 \DeclareMathOperator{\ob}{ob}
\DeclareMathOperator{\colim}{colim}
 
\newcommand{\cat}[1]{\mathbf{#1}}

\newcommand{\op}{\mathrm{op}}

\newcommand{\id}{\mathrm{id}}
\newcommand{\thg}{{\mathord{\text{--}}}}

\newcommand{\abs}[1]{{\left|{#1}\right|}}
\newcommand{\dbr}[1]{\left\llbracket{#1}\right\rrbracket}
\newcommand{\res}[2]{\left.{#1}\right|_{#2}}
\newcommand{\set}[2]{\left\{#1 \ \vrule\  #2\right\}}

\newcommand{\defeq}{\mathrel{\mathop:}=}

\newcommand{\cd}[2][]{\vcenter{\hbox{\xymatrix#1{#2}}}}

\newcommand{\B}{{\mathcal B}}
\newcommand{\C}{{\mathcal C}}

\newcommand{\E}{{\mathcal E}}
\newcommand{\F}{{\mathcal F}}

\renewcommand{\O}{{\mathcal O}}
\renewcommand{\P}{{\mathcal P}}

\newcommand{\X}{{\mathcal X}}

\newcommand{\xtor}[1]{\cdl[@1]{{} \ar[r]|-{\object@{|}}^{#1} & {}}}

\makeatletter

\def\hookleftarrowfill@{\arrowfill@\leftarrow\relbar{\relbar\joinrel\rhook}}
\def\twoheadleftarrowfill@{\arrowfill@\twoheadleftarrow\relbar\relbar}
\def\leftbararrowfill@{\arrowdoublefill@{\leftarrow\mkern-5mu}\relbar\mapstochar\relbar\relbar}
\def\Leftbararrowfill@{\arrowdoublefill@{\Leftarrow\mkern-2mu}\Relbar\Mapstochar\Relbar\Relbar}
\def\leftringarrowfill@{\arrowdoublefill@{\leftarrow\mkern-3mu}\relbar{\mkern-3mu\circ\mkern-2mu}\relbar\relbar}
\def\lefttriarrowfill@{\arrowfill@{\mathrel\triangleleft\mkern0.5mu\joinrel\relbar}\relbar\relbar}
\def\Lefttriarrowfill@{\arrowfill@{\mathrel\triangleleft\mkern1mu\joinrel\Relbar}\Relbar\Relbar}

\def\hookrightarrowfill@{\arrowfill@{\lhook\joinrel\relbar}\relbar\rightarrow}
\def\twoheadrightarrowfill@{\arrowfill@\relbar\relbar\twoheadrightarrow}
\def\rightbararrowfill@{\arrowdoublefill@{\relbar\mkern-0.5mu}\relbar\mapstochar\relbar\rightarrow}
\def\Rightbararrowfill@{\arrowdoublefill@{\Relbar\mkern-2mu}\Relbar\Mapstochar\Relbar\Rightarrow}
\def\rightringarrowfill@{\arrowdoublefill@\relbar\relbar{\mkern-2mu\circ\mkern-3mu}\relbar{\mkern-3mu\rightarrow}}
\def\righttriarrowfill@{\arrowfill@\relbar\relbar{\relbar\joinrel\mkern0.5mu\mathrel\triangleright}}
\def\Righttriarrowfill@{\arrowfill@\Relbar\Relbar{\Relbar\joinrel\mkern1mu\mathrel\triangleright}}

\def\leftrightarrowfill@{\arrowfill@\leftarrow\relbar\rightarrow}
\def\mapstofill@{\arrowfill@{\mapstochar\relbar}\relbar\rightarrow}

\renewcommand*\xleftarrow[2][]{\ext@arrow 20{20}0\leftarrowfill@{#1}{#2}}
\providecommand*\xLeftarrow[2][]{\ext@arrow 60{22}0{\Leftarrowfill@}{#1}{#2}}
\providecommand*\xhookleftarrow[2][]{\ext@arrow 10{20}0\hookleftarrowfill@{#1}{#2}}
\providecommand*\xtwoheadleftarrow[2][]{\ext@arrow 60{20}0\twoheadleftarrowfill@{#1}{#2}}
\providecommand*\xleftbararrow[2][]{\ext@arrow 10{22}0\leftbararrowfill@{#1}{#2}}
\providecommand*\xLeftbararrow[2][]{\ext@arrow 50{24}0\Leftbararrowfill@{#1}{#2}}
\providecommand*\xleftringarrow[2][]{\ext@arrow 10{26}0\leftringarrowfill@{#1}{#2}}
\providecommand*\xlefttriarrow[2][]{\ext@arrow 80{24}0\lefttriarrowfill@{#1}{#2}}
\providecommand*\xLefttriarrow[2][]{\ext@arrow 80{24}0\Lefttriarrowfill@{#1}{#2}}

\renewcommand*\xrightarrow[2][]{\ext@arrow 01{20}0\rightarrowfill@{#1}{#2}}
\providecommand*\xRightarrow[2][]{\ext@arrow 04{22}0{\Rightarrowfill@}{#1}{#2}}
\providecommand*\xhookrightarrow[2][]{\ext@arrow 00{20}0\hookrightarrowfill@{#1}{#2}}
\providecommand*\xtwoheadrightarrow[2][]{\ext@arrow 03{20}0\twoheadrightarrowfill@{#1}{#2}}
\providecommand*\xrightbararrow[2][]{\ext@arrow 01{22}0\rightbararrowfill@{#1}{#2}}
\providecommand*\xRightbararrow[2][]{\ext@arrow 04{24}0\Rightbararrowfill@{#1}{#2}}
\providecommand*\xrightringarrow[2][]{\ext@arrow 01{26}0\rightringarrowfill@{#1}{#2}}
\providecommand*\xrighttriarrow[2][]{\ext@arrow 07{24}0\righttriarrowfill@{#1}{#2}}
\providecommand*\xRighttriarrow[2][]{\ext@arrow 07{24}0\Righttriarrowfill@{#1}{#2}}

\providecommand*\xmapsto[2][]{\ext@arrow 01{20}0\mapstofill@{#1}{#2}}
\providecommand*\xleftrightarrow[2][]{\ext@arrow 10{22}0\leftrightarrowfill@{#1}{#2}}
\providecommand*\xLeftrightarrow[2][]{\ext@arrow 10{27}0{\Leftrightarrowfill@}{#1}{#2}}

\makeatother

\newcommand{\twocong}[2][0.5]{\ar@{}[#2] \save ?(#1)*{\cong}\restore}
\newcommand{\twoeq}[2][0.5]{\ar@{}[#2] \save ?(#1)*{=}\restore}
\newcommand{\rtwocell}[3][0.5]{\ar@{}[#2] \ar@{=>}?(#1)+/l 0.2cm/;?(#1)+/r 0.2cm/^{#3}}
\newcommand{\ltwocell}[3][0.5]{\ar@{}[#2] \ar@{=>}?(#1)+/r 0.2cm/;?(#1)+/l 0.2cm/^{#3}}
\newcommand{\ltwocello}[3][0.5]{\ar@{}[#2] \ar@{=>}?(#1)+/r 0.2cm/;?(#1)+/l 0.2cm/_{#3}}
\newcommand{\dtwocell}[3][0.5]{\ar@{}[#2] \ar@{=>}?(#1)+/u  0.2cm/;?(#1)+/d 0.2cm/^{#3}}
\newcommand{\dltwocell}[3][0.5]{\ar@{}[#2] \ar@{=>}?(#1)+/ur  0.2cm/;?(#1)+/dl 0.2cm/^{#3}}
\newcommand{\urtwocell}[3][0.5]{\ar@{}[#2] \ar@{=>}?(#1)+/dl  0.2cm/;?(#1)+/ur 0.2cm/^{#3}}
\newcommand{\drtwocell}[3][0.5]{\ar@{}[#2] \ar@{=>}?(#1)+/ul  0.2cm/;?(#1)+/dr 0.2cm/^{#3}}
\newcommand{\dthreecell}[3][0.5]{\ar@{}[#2] \ar@3{->}?(#1)+/u  0.2cm/;?(#1)+/d 0.2cm/^{#3}}
\newcommand{\utwocell}[3][0.5]{\ar@{}[#2] \ar@{=>}?(#1)+/d 0.2cm/;?(#1)+/u 0.2cm/_{#3}}
\newcommand{\dtwocelltarg}[3][0.5]{\ar@{}#2 \ar@{=>}?(#1)+/u  0.2cm/;?(#1)+/d 0.2cm/^{#3}}
\newcommand{\utwocelltarg}[3][0.5]{\ar@{}#2 \ar@{=>}?(#1)+/d  0.2cm/;?(#1)+/u 0.2cm/_{#3}}

\newcommand{\pushoutcorner}[1][dr]{\save*!/#1-1.2pc/#1:(-1,1)@^{|-}\restore}

\newdir{(}{{}*!<0em,-.14em>-\cir<.14em>{l^r}}
\newdir{ (}{{}*!/-5pt/\dir{(}}
\newdir{ >}{{}*!/-5pt/\dir{>}}

\swapnumbers
\theoremstyle{plain}
\newtheorem{Thm}{Theorem}[section]
\newtheorem{Prop}[Thm]{Proposition}

\newtheorem*{Lemma}{Lemma}

\theoremstyle{definition}
\newtheorem{Defn}[Thm]{Definition}

\theoremstyle{remark}
\newtheorem{Ex}[Thm]{Example}
\newtheorem{Exs}[Thm]{Examples}
\newtheorem{Rk}[Thm]{Remark}

\DeclareMathOperator{\Lan}{Lan}

\begin{document}
 \leftmargini=2em
\title{Ionads}
\author{Richard Garner}
\address{Department of Computing, Macquarie University, Sydney NSW 2109, Australia}
\email{richard.garner@mq.edu.au} \subjclass[2010]{Primary: 18B25, 54A05}
\date{\today}
\begin{abstract}
The notion of Grothendieck topos may be considered as a generalisation of that
of topological space, one in which the points of the space may have non-trivial
automorphisms. However, the analogy is not precise, since in a topological
space, it is the points which have conceptual priority over the open sets,
whereas in a topos it is the other way around. Hence a topos is more correctly
regarded as a generalised locale, than as a generalised space. In this article
we introduce the notion of \emph{ionad}, which stands in the same relationship
to a topological space as a (Grothendieck) topos does to a locale. We develop
basic aspects of their theory and discuss their relationship with toposes.
\end{abstract}
\maketitle

\section{Introduction} Grothendieck introduced \emph{toposes}
in \cite{SGA4.1} in order to describe a more general kind of ``space'' than
that given by general topology, one whose ``points'' could possess non-trivial
automorphisms. However, as Grothendieck himself immediately points out, the
notion of topos is not a faithful generalisation of that of topological space;
for though each space gives rise to a topos---namely the topos of sheaves on
that space---we may only reconstruct the space from the topos if the space
satisfies a separability axiom (``sobriety''). This is a reflection of a more
general fact concerning the continuous maps between spaces. Every such map
induces a geometric morphism between the corresponding sheaf toposes, and so we
have a functor $\cat{Sp} \to \cat{GTop}$ from the category of spaces to the
category of Grothendieck toposes; but this functor is neither full nor faithful
(not even in the bicategorical sense). As is well known, the reason for these
discrepancies is that a Grothendieck topos is not really a more general kind of
\emph{space}, but rather a more general kind of
\emph{locale}~\cite{Ehresmann1957Gattungen,Johnstone1982Stone}. Indeed, to
every locale we may assign a topos---again, by the sheaf construction---but, by
contrast with the case for spaces, it is always possible to reconstruct the
locale from the topos. We obtain similar good behaviour with respect to
continuous maps of locales: the sheaf construction extends to a functor
$\cat{Loc} \to \cat{GTop}$ which is both (bicategorically) full and faithful,
and has a (bicategorical) left adjoint.

The reason that we bring this up is to point out a gap in our conceptual
framework: there is no established structure which generalises the notion of
topological space in a manner corresponding to that in which a Grothendieck
topos generalises a locale. The purpose of this paper is to fill this gap by
introducing the notion of \emph{ionad}\footnote{``ionad'' is a Gaelic word
meaning ``place''}. Like a topological space, an ionad comprises a set of
points together with a ``topology''. The notion of topology employed is not the
classical one, but is at least a generalisation of it: which in particular
means that there is a canonical way of viewing a topological space as an ionad,
giving rise to a full, reflective embedding of the category of spaces into the
category of ionads. Thus an ionad is indeed a ``generalised space'', and the
tightness of the correspondence is confirmed by many further correlations
between the theory of topological spaces and that of ionads: for instance, a
continuous map between ionads is a function on the underlying sets which
commutes with the topologies in an appropriate sense; (co)limits in the
category of ionads are constructed by equipping the (co)limit of the underlying
diagram of sets with a suitably universal topology; and the most common way of
constructing an ionad is in terms of a set of points together with a
(suitably-generalised) basis of opens. Moreover, any ionad has a collection of
``opens'', and just as the opens of a topological space form a locale, so the
opens of an ionad form a topos. For those ionads arising from topological
spaces, this ``topos of opens'' is just the topos of sheaves on the space;
which is to say that the passage from spaces to ionads to toposes coincides
with that from spaces to locales to toposes.
%
%
%Any topological space may be viewed as
%%As for topological
%%spaces, the continuous maps between ionads are functions on the underlying sets
%%which respect the topologies; and a
%As for topological spaces, we may assign to each ionad a collection of
%``opens'': but whereas the opens of a space form a locale, the opens of an
%ionad form a \emph{topos}. Moreover, the category of topological spaces embeds
%fully and reflectively into the category of ionads, just as the category of
%locales embeds fully and reflectively into the category of Grothendieck
%toposes.
%\begin{equation*}
%    \cd{
%        \cat{Ion} \ar[r] & \cat{Top} \\
%        \cat{Sp} \ar[r] \ar[u] & \cat{Loc}\rlap{ .} \ar[u]
%    }
%\end{equation*}
%Now just as locales embed fully and reflectively into the category of
%Grothendieck toposes, so topological spaces embed fully and reflectively into
%the category of ionads.

If we do take a topos-theoretic perspective, then the notion of ionad turns out
to have a succinct and familiar expression---it is nothing other than a topos
equipped with a separating set of points. Many other aspects of ionad theory
admit similarly familiar topos-theoretic interpretations. However, the view we
take here is that it should be perfectly possible to develop the basic theory
of ionads without reducing it to the theory of toposes, and we have arranged
our account accordingly.
%
%
% However, it is equally well the case that a
%topological space is nothing other than a locale equipped with a separating set
%of points.
%
%
%
%
%
%With this in mind, it seems unlikely that ionads will allow us to prove
%anything which cannot already be proven using toposes; in fact, it seems likely
%that in most applications, ionads will be inferior to toposes, for essentially
%the same reasons that topological spaces are often found inferior to locales.
%Nonetheless, there are at least two good reasons for having the notion of ionad
%available. The first is that there are at least \emph{some} pieces of topos
%theory for which the ionad-theoretic perspective is a very natural one: for
%instance, one could usefully recast Awodey and Kinosha's sheaf-theoretic
%interpretation of first-order modal logic, or Forssell's generalised Stone
%duality~\cite{}, in these terms. The second reason is pedagogical: there are
%some aspects of topos theory that may be explained more perspicuously by
%generalising from spaces to ionads to toposes, than by generalising from spaces
%to locales to toposes: for example, the passage from the space of
%(elementary-equivalence classes of) complete models of a first-order theory to
%the classifying topos of that theory. We have therefore sought to present
%ionads and associated notions in a manner which does not assume familiarity
%with their topos-theoretic correlates (though we always attempt to make the
%link clear).

The paper is structured as follows. In Section~\ref{s1}, we introduce the
notion of ionad, and give such examples as we may construct with our bare
hands. We see that the ``generalised opens'' of an ionad are always a topos,
and that, as mentioned above, the notion of ionad is essentially the same as
that of spatial topos. Then in Section~\ref{s2} we describe how ionads may be
generated from a \emph{basis}; where this notion naturally generalises the
corresponding one for spaces. Using this we are able to give several more
examples of ionads, which correspond to some familiar examples of toposes. We
investigate the connection between the ionad generated from a basis, and the
Grothendieck topos generated from a site, and finally characterise the ionads
which may be generated from a basis (which we call \emph{bounded}) as being
those whose category of opens is a Grothendieck topos.

In Section~\ref{s3} we define \emph{continuous maps} between ionads and give a
number of examples. We shall see that both the category of topological spaces
and the category of small categories embed fully and faithfully into the
category of bounded ionads, with the former embedding having a left adjoint;
later, we show that the latter one has a right adjoint. In Sections~\ref{s5}
and~\ref{lims}, we briefly consider further aspects of the theory.
Section~\ref{s5} shows that the category of ionads may be enriched to a
$2$-category, whose $2$-cells are \emph{specialisations} between continuous
maps, generalising the (pointwise) specialisation order on maps of topological
spaces; and Section~\ref{lims} describes the limits and colimits possessed by
the $2$-category of ionads: we see that the $2$-category of \emph{all} ionads
has rather few limits and colimits, but that the $2$-category of bounded ionads
is complete and cocomplete. The paper is concluded in Section~\ref{s6} by a
short discussion on the comparative advantages and disadvantages of the notions
of topos and ionad.
%We see that the category of all ionads has rather few limits and
%colimits, but that upon restricting to the \emph{bounded} ionads, we obtain a
%complete and cocomplete category. We also investigate the relationship between
%limits and colimits of ionads and those of spaces and categories.
%
%Finally,  and show that the $2$-category of bounded ionads, continuous maps
%and specialisations is $2$-categorically complete and cocomplete. We show as a
%consequence of this that the embedding of categories into bounded ionads has a
%right adjoint, which sends a bounded ionad to the ``generalised specialisation
%order'' on its set of points.

\textbf{Acknowledgements}. The material of this paper was first presented at
the 2009 Category Theory conference in Cape Town. I have benefitted from
discussions with Martin Hyland, Steve Lack, Pedro Resende, Emily Riehl, Michael
Shulman, Thomas Streicher and Dom Verity; whilst from an anonymous referee I
received detailed and incisive feedback which has improved the paper
considerably. I was supported in this work by a Research Fellowship of St
John's College, Cambridge and an Australian Research Fellowship.

\section{Ionads}\label{s1}
We usually define a topological space to be a set $X$ of points together with a
\emph{topology}: a collection of subsets of $X$ closed under finite
intersections and arbitrary unions. However, we may equally well give the set
$X$ together with an \emph{interior operator}: an order-preserving map $i
\colon \P X \to \P X$ which is a coclosure operator (i.e., deflationary and
idempotent) and preserves finite intersections. The passage between the two
definitions is straightforward: given a topology on $X$, there is an interior
operator sending $A \subseteq X$ to the largest open set contained in it; and
given an interior operator $i \colon \P X \to \P X$, there is a topology on $X$
consisting of all those $A \subseteq X$ for which $A = i(A)$. Now by taking
this second definition of topological space and replacing every poset-theoretic
device which appears in it with a corresponding
%\footnote{There are choice to be made here: the
%appropriate generalisation of a finite-meet preserving monotone map could
%equally well be a functor preserving finite products, and the appropriate
%generalisation of a coclosure operator could have been an \emph{idempotent}
%comonad. Moreover, our starting point for generalisation could have been a
%space presented in terms of a closure, rather than an interior, operator.
%However, the particular choices we have made are the natural ones in order to
%obtain a theory resembling that of toposes.}
category-theoretic one, we obtain
the notion of ionad.

\begin{Defn}\label{genspace}
An \emph{ionad} is given by a set $X$ of points together with a cartesian
(i.e., finite limit preserving) comonad $I_X \colon \cat{Set}^X \to
\cat{Set}^X$.
\end{Defn}

\begin{Rk}\label{not1}
It will be convenient to carry over from the topological case the abuse of
notation which names a space by its set of points: thus we may refer to an
ionad $(X, I_X)$ simply as $X$, with the interior comonad being left implicit.
%On those
%occasions where the confusion of the ionad with its set of
%points introduces a genuine ambiguity, we will write $X_0$ to
%denote the set of points, and $(X, I_X)$ to denote the entire
%ionad.
\end{Rk}

If we are given a topological space presented in terms of its interior
operator~$i$, then we can reconstruct its open sets as the collection of
$i$-fixpoints. In the case of an ionad, the only presentation we have is in
terms of a generalised interior operator: but we can use the category-theoretic
analogue
%\footnote{Again, there is a choice here. Viewing an interior operator
%$i$ as a comonad, the collection of $i$-fixpoints is the category of
%$i$-coalgebras; but it is equally well the co-Kleisli category of~$i$.}
of the
fixpoint construction in order to \emph{define} its ``generalised opens''.

\begin{Defn}\label{catopens}
The \emph{category of opens} $\cat O(X)$ of an ionad $X$ is the category of
$I_X$-coalgebras.
\end{Defn}

\begin{Rk}
In the definition of ionad, we have chosen to have a mere \emph{set} of points,
rather than a category of them. We do so for a number of reasons. The first is
that this choice mirrors most closely the definition of topological space,
where we have a set, and not a poset, of points. The second is that we would in
fact obtain no extra generality by allowing a category of points. We may see
this by analogy with the topological case, where to give an interior operator
on a poset of points $(X, \leqslant)$ is equally well to give a topology
$\O(X)$ on $X$ such that every open set is upwards-closed with respect to
$\leqslant$. Similarly, to equip a small category $\cat C$ with an interior
comonad is equally well to give an interior comonad on $X
\defeq \ob \cat C$ together with a factorisation of the
forgetful functor $\cat O(X) \to \cat{Set}^{X}$ through the presheaf category
$\cat{Set}^{\cat C}$; this is an easy consequence of Example~\ref{alexandroff}
below. However, the most compelling reason for not admitting a category of
points is that, if we were to do so, then adjunctions such as that between the
category of ionads and the category of topological spaces would no longer
exist. Note that, although we do not allow a category of points, the points of
any (well-behaved) ionad bear nonetheless a canonical category
structure---described in Definition~\ref{specfunctor} and Remark~\ref{vx}
below---which may be understood as a generalisation of the specialisation
ordering on the points of a space.
%It is through
%this category structure that an ionad may realise Grothendieck's vision of a
%``generalised space'' as one whose points may possess non-trivial
%automorphisms.
\end{Rk}

\begin{Rk}\label{discrete}
As stated in the Introduction, the category of opens of an ionad is always a
(cocomplete, elementary) topos: this because $\cat{Set}^X$ is a topos for any
set $X$, and the category of coalgebras for a cartesian comonad on a topos is
again a topos. Moreover, the cofree/forgetful adjunction between $\cat{Set}^X$
and $\cat{O}(X)$ yields a surjective geometric morphism $\cat{Set}^X \to
\cat{O}(X)$. To give such a geometric morphism is to give an $X$-indexed family
of points of the topos $\cat{O}(X)$; to say that it is surjective is to say
that these points separate the generalised opens in $\cat{O}(X)$, in the sense
that their inverse image functors jointly reflect isomorphisms. In particular,
this makes $\cat{O}(X)$ a topos with \emph{enough points}~\cite[\S
C2.2]{Johnstone2002Sketches2}.

In fact, given any surjective geometric morphism $f \colon \cat{Set}^X \to \E$,
we obtain an ionad $(X, f^\ast f_\ast)$ whose category of open sets is
equivalent (by surjectivity of $f$) to $\mathcal E$. Thus ionads are
essentially the same things as toposes equipped with a separating set of
points. There is a parallel here with the theory of topological spaces: where a
space may be identified with a surjective locale morphism out of a discrete
locale---one of the form $\P X$ for some set $X$, an ionad may be identified
with a surjective geometric morphism out of a discrete topos---one of the form
$\cat{Set}^X$ for some set $X$. \end{Rk}

\begin{Rk}
We shall see in Examples~\ref{ordgen}.2 below that every topological space $A$
gives rise to an ionad $\Sigma A$, and that the topos of opens $\cat O(\Sigma
A)$ is equivalent to the topos of sheaves $\cat{Sh}(A)$. With this in mind, we
could have chosen to refer to the topos $\cat{O}(X)$ as the \emph{topos of
sheaves} on the ``generalised space'' $X$. We will not do so here, for the
following two reasons. The first is that, in generalising further concepts from
topological spaces to ionads, we often need do nothing more than replace
$\O(X)$ everywhere by $\cat O(X)$, and the inevitability of this replacement
would be diminished if we were to refer to this latter topos as $\cat{Sh}(X)$.
The second reason is that, whilst it is indeed true that objects of $\O(X)$
look very much like sheaves on a topological space---as evidenced by
Proposition~\ref{sheaves}, for instance---it is none the less true that they
look very much like generalised open sets. We will expand on this point in
Remark~\ref{rk:genopen} below.
\end{Rk}

\begin{Ex}\label{alexandroff}
If $(X, \leqslant)$ is a partially ordered set, then there is a topology on
$X$---the \emph{Alexandroff topology}---whose open sets are the upwards-closed
subsets of $X$ with respect to $\leqslant$. In a similar way, if ${\cat C}$ is
a small category, then there is an ionad $A(\cat C)$ on the set of objects of
$\cat C$ whose generalised opens are the ``generalised upsets'' in $\cat C$:
that is, the covariant presheaves on $\cat{C}$. The interior comonad of this
ionad is induced by the adjunction
\begin{equation}\label{genalexandroff}
    \cd[@C+1em]{\cat{Set}^{\ob \cat C} \ar@<-4pt>[r]_-{\mathrm{Ran}_J} \ar@{}[r]|{\bot} & {\cat{Set}^{\cat C}} \ar@<-4pt>[l]_-{\cat{Set}^J}}
\end{equation}
obtained by restriction and right Kan extension along the inclusion $J \colon
\ob \cat C \rightarrowtail   {\cat C}$. Observe that the functor $\cat{Set}^J$,
since it strictly creates equalisers, is strictly comonadic; and so the
category of open sets for $A(\cat C)$ is isomorphic to $\cat{Set}^{\cat C}$.

Recall that the lattice of open sets of an Alexandroff topology is closed under
\emph{arbitrary} intersections, and that this property serves to completely
characterise the Alexandroff topologies. Likewise, for an Alexandroff ionad,
the forgetful functor $U \colon \cat O(A(\cat C)) \to \cat{Set}^{\ob \cat C}$
creates limits (since $\cat{Set}^J$, and hence the interior comonad, preserve
them); and this property completely characterises the Alexandroff ionads.
Indeed, if for some ionad $X$ the forgetful functor $\cat O(X) \to \cat{Set}^X$
creates limits, then the interior operator $I_X \colon \cat{Set}^{X} \to
\cat{Set}^{X}$ will preserve them; and so have a colimit-preserving left
adjoint $K$. The comonad structure of $I_X$ transposes across the adjunction to
give a monad structure on $K$, which is equivalently a monoid structure on the
functor
\begin{equation*}
    M \defeq X \xrightarrow{y} \cat{Set}^{X} \xrightarrow{K}  \cat{Set}^{X}
\end{equation*}
with respect to profunctor composition; and this in turn amounts to specifying
a category ${\cat C}$ with object set $X$ and homsets ${\cat C}(x, y)
\defeq M(x)(y)$. Moreover, it follows (``adjoint triples'' \cite{Eilenberg1965Adjoint}) that there is an
isomorphism between the category of $I_X$-coalgebras%
---which is $\cat O(X)$---and the category of $K$-algebras, which, by an easy calculation, is $\cat{Set}^{\cat C}$.
\end{Ex}

\section{Generalised bases}\label{s2}
In order to produce more sophisticated examples of ionads, we will need a way
of generating topologies from bases. Recall that a \emph{basis} for an ordinary
topology on a set $X$ is a collection $\B \subseteq \P X$ satisfying the
following properties:
\begin{itemize}
\item For every $x \in X$, there is some $B \in \B$ with $x \in \B$;
\item If $x \in X$ and $B_1, B_2 \in \B$ with $x \in B_1 \cap B_2$, then
    there is some $B_3 \in \B$ with $B_3 \subseteq B_1 \cap B_2$ and $x \in
    B_3$.
\end{itemize}
The open sets of the topology this generates are arbitrary unions of elements
of $\B$. However, since our aim is to generalise this definition from spaces to
ionads, we will be more interested in describing the interior operator
generated by $\B$. To this end, we regard $\B$ as a poset under inclusion, and
write $m \colon \B \to \P X$ for the (order-preserving) inclusion map. Now the
basis axioms for $\B$ correspond to the requirement that $m$ should be
\emph{flat}, in the sense that, for each $x \in X$, the set $\set{B \in \B}{x
\in m(B)}$ should be a downwards directed poset. In fact, we can drop the
requirement that $m$ should be injective entirely; this gives a more
``intensional'' notion of basis, where the same open set may be named by more
than one basis element. (Observe that this happens quite frequently in
practice: think, for example, of the Zariski topology on the prime spectrum of
a ring; or of the logical topology on the set of complete theories extending a
first-order theory $\mathbb T$.)

Now given any flat morphism $m \colon \B \to \P X$, we may define an interior
operator on $X$ in the following manner. We write $\mathord \downarrow \B$ for
the poset of downsets in $\B$, and $y \colon \B \to \mathord \downarrow \B$ for
the order-preserving map sending $B \in \B$ to the downset of all elements
below $B$. This map exhibits $\mathord \downarrow \B$ as the free
join-completion of $\B$, and so there's a unique way of extending $m$ along $y$
to yield a join-preserving map $m \otimes (\thg) \colon \mathord \downarrow \B
\to \P X$:
\begin{equation*}
    \cd{
        \B \ar[d]_y \ar[r]^-m & \P X \\ \mathord \downarrow \B \ar@{-->}[ur]_{m \otimes (\thg)}
    } \ \text.
\end{equation*}
Since this $m \otimes (\thg)$ preserves joins, it has a right adjoint $[m,
\thg] \colon \P X \to \mathord \downarrow \B$, and composing these yields a
coclosure operator $i \defeq m \otimes [m, \thg]$ on $\P X$. In order to show
that this $i$ preserves finite meets, it suffices to show that $m \otimes \thg$
does so; but a standard piece of lattice theory tells us that this is
equivalent to the flatness of $m$. It remains to check that this interior
operator $i$ is the one associated to the topology generated by $\B$. For this,
we calculate that
\begin{equation*}
    m \otimes \X = \bigcup_{B \in \X} m(B) \qquad \text{and} \qquad [m, A] = \set{B \in \B}{m(B) \subseteq A}
\end{equation*}
so that the composite $i \colon \P X \to \P X$ sends $A$ to the union of all
those $m(B)$'s with $m(B) \subseteq A$, as required. Consideration of the above
now leads us to propose:

\begin{Defn}
A \emph{basis} for an ionad with set of points $X$ is given by a small category
$\cat B$ together with a functor $M \colon \cat B \to \cat{Set}^X$ which is
\emph{flat}, in the sense that for each $x \in X$, the category of elements of
the functor $M(\thg)(x) \colon \cat B \to \cat{Set}$ is cofiltered.
\end{Defn}
The construction of an ionad from a basis mirrors that of a space from a basis.
The Yoneda embedding $y \colon \cat B \to \cat{Set}^{\cat B^\op}$ exhibits
$\cat{Set}^{\cat B^\op}$ as the free colimit-completion of $\cat B$, and so we
may extend $M$ along it to yield a colimit-preserving functor $M \otimes (\thg)
\colon \cat{Set}^{\cat B^\op} \to \cat{Set}^X$:
\begin{equation*}
    \cd{
        \cat B \ar[d]_y \ar[r]^-M \twocong[0.35]{dr} & \cat{Set}^X\ \text. \\ \cat{Set}^{\cat B^\op} \ar@{-->}[ur]_{M \otimes (\thg)} & {}
    }
\end{equation*}
Since this $M \otimes (\thg)$ preserves colimits, it has a right adjoint $[M,
\thg]$, and composing these together yields a comonad on $\cat{Set}^X$. Again,
to ensure that this comonad preserves finite limits, it suffices to show that
$M \otimes (\thg)$ does; and a standard piece of category theory says that this
is equivalent to the flatness of $M$.

\begin{Rk}\label{rk:genopen}
If $\B$ is a basis for an ordinary topology on a set $X$, then a subset $A
\subset X$ is open in that topology just when every $x \in A$ is contained in
some $B \in \B$ with $B \subset A$. If $M \colon \cat B \to \cat{Set}^X$ is a
basis for an ionad $X$, then we may view objects of the category of opens $\cat
O(X)$ in a corresponding manner. Unravelling the definitions, we see that the
interior comonad $I$ generated by the basis $M$ has its value at $A \in
\cat{Set}^X$ given by
\begin{equation*}
    (I A)(x) = \int^{B \in \cat B} (MB)(x) \times \prod_{y \in X}A(y)^{(MB)(y)}\rlap{ .}
\end{equation*}
Thus, if we think of a typical $A \in \cat{Set}^X$ as specifying, for each $x
\in X$, a set $Ax$ of proofs that $x$ lies in $A$, then to give an
$I$-coalgebra structure on $A$ is to give a mapping which, to each proof that
$x$ lies in $A$, coherently assigns an element $B \in \cat B$, together with
proofs that $x$ lies in $MB$ and that $MB$ is contained in $A$.
\end{Rk}

\begin{Rk}\label{factor}
Having motivated the preceding construction purely from topological
considerations, we now see that it is a familiar one in topos-theory. A basis
for an ionad is a flat functor $\cat B \to \cat{Set}^X$, which corresponds to a
colimit-\ and finite-limit-preserving functor $\cat {Set}^{\cat B^\op} \to
\cat{Set}^X$, and hence to a geometric morphism $\cat{Set}^X \to \cat
{Set}^{\cat B^\op}$. Any such geometric morphism factors as
\begin{equation}\label{eq:toposfact}\cat{Set}^X \xrightarrow{\ \ p\ \ } \E
\xrightarrow{\ \ i\ \ } \cat {Set}^{\cat B^\op}\end{equation} where $p$ is a
geometric surjection, and $i$ a geometric inclusion (see~\cite[Theorem
A4.2.10]{Johnstone2002Sketches} for example). By Remark~\ref{discrete}, the map
$p$ determines an ionad on $X$, which is by inspection precisely the ionad
generated by the basis $\cat B \to \cat{Set}^X$. The fact that $i$ is a
geometric inclusion tells us that $\E$, the category of opens of this ionad, is
a subtopos of $\cat {Set}^{\cat B^\op}$. In fact, we have:
%We shall return to this point after
%giving some examples.
%%
%%in Definition~\ref{assocsite} and Theorem~\ref{sheaves} below.
\end{Rk}
\begin{Prop}\label{sheaves}
If $M \colon \cat B \to \cat{Set}^X$ is a basis for an ionad $X$, then the
category $\cat O(X)$ is equivalent to the category of sheaves on the site whose
underlying category is $\cat B$ and whose covering families are those $(f_i
\colon U_i \to U \mid i \in I)$ in $\cat B$ which $M$ sends to jointly
epimorphic families in $\cat{Set}^X$.
\end{Prop}
\begin{proof}
Since $\cat O(X)$ is the category of coalgebras for the comonad generated by
the adjunction
\begin{equation*}
    \cd{\cat{Set}^{X} \ar@<-4pt>[r]_{[M, \thg]} \ar@{}[r]|{\bot} & {\cat{Set}^{\cat B^\op}} \ar@<-4pt>[l]_{M \otimes (\thg)}}\ \text,
\end{equation*}
there is a canonical comparison functor $L \colon \cat{Set}^{\cat B^\op} \to
\cat O(X)$, which preserves finite limits since $M \otimes (\thg)$ does.
Because $\cat{Set}^{\cat{B}^\op}$ has equalisers, this functor has a right
adjoint $R$; because $M \otimes (\thg)$ preserves them, this $R$ is fully
faithful, and so exhibits $\cat O(X)$, up-to-equivalence, as the category of
sheaves for a site structure on $\cat B$. A family of morphisms $(f_i \colon
U_i \to U \mid i \in I)$ is covering for this site just when the subobject
$\phi \colon A \rightarrowtail \cat B(\thg,U)$ defined by $A(V) = \{\,p \colon
V \to U \mid p\text{ factors through some $f_i$}\,\}$ is sent to an isomorphism
by $L$ (cf.~\cite[Example A4.3.5]{Johnstone2002Sketches}). Since the forgetful
functor $U_X \colon \cat O(X) \to \cat{Set}^X$ is comonadic, hence
conservative, to ask that $L\phi$ is invertible in $\cat O(X)$ is equally well
to ask that $U_XL\phi = M \otimes \phi$ be so in $\cat{Set}^X$. But we may
obtain $\phi$ as the second half of the regular epi-mono factorisation of the
map $[\C(\thg, f_i)]_{i \in I} \colon \sum_{i \in I} \C(\thg,U_i) \to \C(\thg,
U)$; and since $M\otimes(\thg)$ preserves finite limits and small colimits, it
preserves such factorisations, whence $M \otimes \phi$ is an isomorphism if and
only if $(M \otimes \C(\thg, f_i)  \mid i \in I)$ is a jointly epimorphic
family in $\cat{Set}^X$; which is equally well if and only if $(Mf_i \mid i \in
I)$ is a jointly epimorphic family, as claimed.
\end{proof}

\begin{Exs}\label{ordgen}\hfill
\begin{enumerate}
\item If $\cat C$ is a small category, then we obtain a basis for the
    Alexandroff ionad $A(\cat C)$ of Example~\ref{alexandroff} by taking
    $\cat B
    \defeq \cat C^\op$ and $M \defeq [J, \thg] \colon \cat
    C^\op \to \cat{Set}^{\ob \cat C}$, where $J \colon \ob \cat C \to
    \cat{C}^\op$ is the canonical inclusion. We see that the ionad this
    basis generates is $A(\cat C)$ by noting that the adjunction $M \otimes
    (\thg) \dashv [M, \thg]$ induced by $M$ is precisely the adjunction
    of~\eqref{genalexandroff}. \vskip0.5\baselineskip
\item Every topological space $X$ gives rise to an ionad $\Sigma X$ on the
    same set of points, generated by the following basis. We take $\cat B
    \defeq \O(X)$, the lattice of open sets of the topology
    (though we could equally well take it to be any basis, in the classical
    sense, for the topology on $X$), and $M \colon \O(X) \to \cat{Set}^X$
    the composite of $\O(X) \rightarrowtail \P X$ with the obvious
    inclusion $\P X \rightarrowtail \cat{Set}^X$. Thus we have
\begin{equation*}
    M(U)(x) = \begin{cases}1 & \text{if $x \in U$;} \\0 & \text{otherwise.}\end{cases}
\end{equation*}
    Clearly $M$ preserves finite limits, and so is flat; and it's easy to
    see that the covering families of the induced site structure on $\cat
    B$ are of the form $(U_i \subseteq U \mid i \in I)$ where $\bigcup U_i
    = U$. Thus by Proposition~\ref{sheaves}, the category of opens $\cat
    O(\Sigma X)$ is equivalent to the category of sheaves on the space $X$.
    \vskip0.5\baselineskip
\vskip0.5\baselineskip
\item Let $X$ be a topological space equipped with an action by a discrete
    group~$G$. We define the \emph{$G$-equivariant ionad} $\Sigma_G X$ to
    have set of points $X$, and topology generated by the following basis.
    We take $\cat B \defeq \O_G(X)$, the category whose objects are open
    sets of $X$, and whose morphisms $U \to V$ are elements $g \in G$ for
    which $g(U) \subseteq V$; and define $M \colon \O_G(X) \to \cat{Set}^X$
    by
    \begin{equation*}
        M(U)(x) = \set{h \in G}{hx \in U} \qquad \text{and} \qquad M(g)(x) \colon h \mapsto gh\ \text.
    \end{equation*}
    It's easy to show that $M$ is flat, and so defines a basis for an ionad
    $\Sigma_G X$. The induced site structure on $\O_G(X)$ has as covering
    families all those $(g_i \colon U_i \to V \mid i \in I)$ such that
    $(Mg_i)$ is jointly epimorphic; that is, such that for every $x \in X$
    and $h \in G$ with $hx \in V$, there exists $i \in I$ such that
    $g_i^{-1}hx \in U_i$; that is, such that the family of maps
    $\res{g_i}{U_i} \colon U_i \to X$ jointly cover $V$. Thus by
    Proposition~\ref{sheaves} and~\cite[Examples
    A.2.1.11(c)]{Johnstone2002Sketches}, we conclude that $\cat O(\Sigma_G
    X)$ is equivalent to the topos of $G$-equivariant sheaves on $X$.
    \vskip0.5\baselineskip
\item If $A$ is a commutative ring, we define its \emph{\'etale ionad} as
    follows. Its set of points is $\mathrm{Spec}(A)$, the set of prime
    ideals of $A$, whilst its topology is generated by the following basis.
    The category $\cat B$ is (a skeleton of) $\cat{Et}_A^\op$, the opposite
    of the category of \'etale $A$-algebras, whilst $M \colon
    \cat{Et}_A^\op \to \cat{Set}^{\mathrm{Spec}(A)}$ sends an \'etale
    $A$-algebra $f \colon A \to B$ and a prime ideal $P \lhd A$ to the set
    of all prime ideals $Q \lhd B$ for which $f^{-1}(Q) = P$. The induced
    site structure on $\cat{Et}_A^\op$ has as covering families those $(f_i
    \colon B_i \leftarrow B \mid i \in I)$ such that $(Mf_i \mid i \in I)$
    is jointly epimorphic; but since the disjoint union of the sets
    $(MB)(P)$, as $P$ ranges over the prime ideals of $A$, is clearly the
    set of all prime ideals of $B$, to say that the family $(Mf_i \mid i
    \in I)$ is jointly epimorphic is equally well to say that every prime
    ideal of $B$ is the inverse image of a prime ideal of some $B_i$.
%    and we claim that this is equivalent to the requirement that every
%    prime ideal of $B$ should be the inverse image of a prime ideal of some
%    $B_i$. Indeed, if $(Mf_i \mid i \in I)$ is jointly epimorphic, then for
%    every prime ideal $Q \lhd B$, we have on taking $P \lhd A$ to be the
%    inverse image of $Q$, that $Q \in M(B)(P)$. Thus by surjectivity of
%    $(Mf_i)$, there is some $R  \in M(B_i)(P)$ which is mapped by
%    $M(f_i)(P)$ to $Q$. But this says that $f_i^{-1}(R) = Q$ as required.
%    Conversely, if every prime ideal of $B$ is the inverse image of a prime
%    ideal of some $B_i$, then for any $P \in \mathrm{Spec}(A)$ and $Q \in
%    M(B)(P)$, we have---since $Q$ is prime in $B$---some $f_i \colon B \to
%    B_i$ and $R$ prime in $B_i$ such that $Q = f_i^{-1}(R)$. But now we
%    have $R \in M(B_i)(P)$ which is mapped by $M(f_i)(P)$ to $Q$, so that
%    $(Mf_i)$ is jointly epimorphic as claimed.
    Thus by Proposition~\ref{sheaves} and~\cite[Exercise
    0.11]{Johnstone1977Topos}, the topos of opens of the \'etale ionad is
    the little \'etale topos of $A$.\vskip0.5\baselineskip
\item Let $\mathbb T$ be a coherent first-order theory over a language of
    cardinality $\kappa$. For any regular cardinal $\lambda > \kappa$, we
    define the \emph{$\lambda$-small classifying ionad} as follows.
 Let $X$ be a set of representatives of isomorphism classes of models of
    $\mathbb T$ of cardinality $< \lambda$, and let $\cat B$ be the
    \emph{syntactic category}~\cite[\S D1.4]{Johnstone2002Sketches2} of the
    theory $\mathbb T$; it has as objects, coherent formulae-in-context
    $\{\vec x. \phi\}$, and as morphisms, equivalence classes of provably
    functional relations between them. We define a functor $M \colon \cat B
    \to \cat{Set}^X$ which takes a formula-in-context $\{\vec x. \phi\}$
    and a model $A$ and returns the interpretation $\dbr{\phi}_A$ of $\phi$
    in~$A$. We may show that $\cat B$ has, and $M$ preserves, all finite
    limits, so that $M$ is a basis for an ionad. Moreover, if
    $\big(\theta_i \colon \{\vec y_i. \phi_i\} \to \{\vec x. \psi\} \mid 1
    \leqslant i \leqslant n)$ is a family of maps in $\cat B$, then it is
    easy to see that the sequent $\psi \vdash_{\vec x} \bigvee_{i=1}^n
    (\exists \vec y_i)\theta_i$ is validated in a given model $A$ of
    $\mathbb T$ if and only if the family of functions $\dbr{\theta_i}_A
    \colon \dbr{\phi_i}_A \to \dbr{\psi}_A$ is jointly epimorphic. Since
    the collection of models of cardinality $< \lambda$ is complete for
    $\mathbb T$, it follows that the family $(\theta_i \mid i \in I)$ is
    sent to a jointly epimorphic family in $\cat{Set}^X$ if and only if the
    sequent $\psi \vdash_{\vec x} \bigvee_{i=1}^n (\exists \vec
    y_i)\theta_i$ is provable in $\mathbb T$: so that the induced site of
    this ionad is the syntactic site of $\mathbb T$, and the topos of
    opens, the classifying topos of $\mathbb T$.
\end{enumerate}
\end{Exs}
Although every ionad we meet in practice will be generated from a basis, it is
not \emph{a priori} the case that every ionad need arise in this way. Indeed,
if an ionad is generated by a basis $M \colon \cat B \to \cat{Set}^X$, then its
category of opens is equivalent to a category of sheaves, and hence a
Grothendieck topos. On the other hand, the category of coalgebras for an
\emph{inaccessible} cartesian comonad on $\cat{Set}^X$ is not locally
presentable, and hence not a Grothendieck topos. In fact, we have:

\begin{Prop}\label{boundprop}
The following conditions on an ionad $X$ are equivalent:
\begin{enumerate}
\item It may be generated (up to isomorphism) from a basis;
\item Its interior comonad is accessible;
\item Its category of opens is a Grothendieck topos.
\end{enumerate}
\end{Prop}
\begin{proof}
First we show (1) $\Rightarrow$ (2). Given a basis $M \colon \cat{B} \to
\cat{Set}^X$, we will show that $[M, \thg] \colon \cat{Set}^X \to
\cat{Set}^{\cat B^\op}$ preserves $\lambda$-filtered colimits for some regular
cardinal $\lambda$; it then follows that also its composite with the left
adjoint $M \otimes (\thg)$ will do so. So let $\lambda$ be such that the set of
objects $\{MB \mid B \in \cat B\}$ are all $\lambda$-presentable in
$\cat{Set}^X$ (such a $\lambda$ exists since $\cat{Set}^X$ is locally
presentable). Now given a $\lambda$-filtered diagram $A \colon \cat I \to
\cat{Set}^X$ we calculate that $[M, \colim_i A_i](B)  = \cat{Set}^X(MB,
\colim_i A_i) \cong \colim_i \cat{Set}^X(MB, A_i)
 = \colim_i [M, A_i](B)$
for every $B \in \cat B$ as required. For (2) $\Rightarrow$ (3), we simply
observe that if $I_X$ is an accessible comonad, then its category of coalgebras
$\cat O(X)$ is locally presentable, and hence a Grothendieck topos. Finally we
show that (3) $\Rightarrow$ (1). If $\cat O(X)$ is a Grothendieck topos, then
it is in particular locally presentable; and so we may find a small full dense
subcategory $N \colon \cat B \rightarrowtail \cat O(X)$ which, without loss of
generality, we may take to be closed under finite limits. Then the composite $M
\defeq \cat B \rightarrowtail \cat O(X) \rightarrow \cat{Set}^X$ is a
cartesian functor on $\cat B$, and so gives a basis for an ionad. The interior
comonad of this ionad is isomorphic to that generated by the string of
adjunctions
\begin{equation*}
    \cd{\cat{Set}^{X} \ar@<-4pt>[r]_{\text{cofree}} \ar@{}[r]|{\bot} & {\cat{O}(X)} \ar@<-4pt>[r]_{[N, \thg]} \ar@{}[r]|{\bot} \ar@<-4pt>[l]_{\text{forget}} & {\cat{Set}^{\cat B^\op}} \ar@<-4pt>[l]_{N \otimes (\thg)}}\ \text,
\end{equation*}
but since $\cat B$ is dense in $\cat O(X)$, the counit of the right-hand
adjunction is an isomorphism, and it follows that the interior comonad of the
resultant ionad is isomorphic to $I_X$.
\end{proof}
\begin{Defn}\label{bounddef}
We call an ionad \emph{bounded} if it satisfies any one of the three equivalent
conditions of Proposition~\ref{boundprop}.
\end{Defn}
\begin{Rk}
Every ionad that we meet in mathematical practice is bounded: moreover, it is
quite probable that the existence or otherwise of unbounded ionads is a problem
that is independent of the usual axioms of set theory. A construction given in
\cite[Example B3.1.12]{Johnstone2002Sketches} shows that any inaccessible
cartesian endofunctor of $\cat{Set}$ gives rise to an unbounded ionad on the
two-element set; but the only known construction of an inaccessible, cartesian
endofunctor of $\cat{Set}$, given in~\cite{Blass1976Exact}, requires the
existence of a proper class of measurable cardinals.
\end{Rk}
\begin{Rk}
In Remark \ref{discrete}, we noted that the category of opens of an ionad is
always a topos with enough points. For bounded ionads, we can say more: the
toposes arising as their categories of opens are \emph{precisely} the
Grothendieck toposes with enough points. Indeed, for a topos $\E$ to have
enough points is for the class of all inverse image functors $\E \to \cat{Set}$
to be jointly conservative. In the case of a Grothendieck topos, this implies
the existence of a mere \emph{set} $X$ of inverse image functors with this
property~(as is shown in \cite[Proposition C2.2.12]{Johnstone2002Sketches2}),
and hence of a geometric surjection $\cat{Set}^X \to \E$ exhibiting $\E$ as the
category of opens of an ionad.
\end{Rk}

\section{Maps of ionads}\label{s3}
We now consider the appropriate notion of morphism between ionads. For ordinary
topological spaces $X$ and $Y$, a continuous map is a morphism of underlying
sets $f \colon X \to Y$ such that the induced inverse image mapping $f^{-1}
\colon \P Y \to \P X$ maps open sets to open sets; which is to say that there
exists a (necessarily unique) lifting
\begin{equation*}
    \cd{
        \O(Y) \ar@{-->}[r]^{f^\ast} \ar@{ >->}[d] & \O(X) \ar@{ >->}[d] \\
        \P Y \ar[r]_{f^{-1}} & \P X
   }
\end{equation*}
of $f^{-1}$ as indicated. We are therefore led to propose:

\begin{Defn}\label{ctsmapdef}
A \emph{continuous map} of ionads $X \to Y$ is a function $f \colon X \to Y$ of
the underlying sets together with a lifting
\begin{equation}\label{mapdef}
    \cd{
        \cat O(Y) \ar[r]^{f^\ast} \ar[d]_{U_Y} & \cat O(X) \ar[d]^{U_X} \\
        \cat{Set}^Y \ar[r]_{f^{-1}} & \cat{Set}^X
   }
\end{equation}
of $f^{-1} \defeq \cat{Set}^f$ through the corresponding categories of open
sets. We write $\cat{Ion}$ for the category of ionads and continuous maps, and
$\cat{BIon}$ for the full subcategory determined by the bounded ionads.
\end{Defn}
\begin{Rk}
In accordance with the notational convention established in Remark~\ref{not1},
we may choose to denote a continuous map of ionads by naming its underlying map
of sets $f \colon X \to Y$ whilst leaving the corresponding lifting $f^\ast
\colon \cat O(Y) \to \cat O(X)$ implicit.
%Again, we will write $f_0$ if we wish to emphasise that it is
%the map of underlying sets that is intended, and write $(f,
%f^\ast)$ to indicate that it is the whole ionad morphism.
\end{Rk}
\begin{Rk}\label{ctstopos}
Observe that if $f$ is an ionad morphism, then $f^\ast \colon \cat O(Y) \to
\cat O(X)$ is cartesian, because finite limits are preserved by $f^{-1}.U_Y$
and reflected by $U_X$. Moreover, since $U_Y$ and $U_X$ are comonadic and $\cat
O(Y)$ has equalisers, the adjoint lifting theorem~\cite{Johnstone1975Adjoint}
permits us to lift the right adjoint $\Pi_f$ of $f^{-1}$ to a right adjoint
$f_\ast$ for $f^\ast$, and so to make $f^\ast$ into the inverse image part of a
geometric morphism $\cat O(X) \to \cat O(Y)$. This construction yields a
functor\footnote{With the usual definition of geometric morphism, this is
really only a \emph{pseudo}functor, because we must choose a right adjoint
$f_\ast$ for each $f^\ast$, and in general cannot expect these choices to
satisfy $g_\ast f_\ast = (gf)_\ast$, but only $g_\ast f_\ast \cong (gf)_\ast$.
However, we shall take the slightly non-standard definition of a geometric
morphism $\E \to \F$ as a left adjoint cartesian functor $\F \to \E$; whereupon
we do indeed obtain a genuine functor $\cat{Ion} \to \cat{Top}$} $\cat O(\thg)
\colon \cat{Ion} \to \cat{Top}$, which is analogous to the functor $\O(\thg)
\colon \cat{Sp} \to \cat{Loc}$ assigning to every topological space its locale
of open sets.
\end{Rk}

\begin{Ex}\label{alexmaps}
If ${\cat C}$ and $\cat D$ are small categories, then to give a continuous map
$A(\cat C) \to A(\cat D)$ between the corresponding Alexandroff ionads is to
give a function $f \colon \ob \cat C \to \ob \cat D$ together with a lifting
\begin{equation}\label{lift}
    \cd{
        \cat{Set}^{\cat D} \ar[r]^{f^\ast} \ar[d] & \cat{Set}^{\cat C} \ar[d] \\
         \cat{Set}^{\ob \cat D} \ar[r]_{f^{-1}} & \cat{Set}^{\ob \cat C}\rlap{ .}
   }
\end{equation}
In particular, any extension of $f$ to a functor $F \colon \cat C \to \cat D$
determines such a lifting by taking $f^\ast = \cat{Set}^F$, and so the
assignation $\cat C \mapsto A(\cat C)$ extends to a functor $A \colon \cat{Cat}
\to \cat{BIon}$. In fact, this functor is fully faithful. To see this, we must
prove that every lifting in~\eqref{lift} is induced by a unique extension of
$f$ to a functor. So given such a lifting $f^\ast$, we must construct for each
$g \colon x \to x'$ in ${\cat C}$ a morphism $F(g) \colon fx \to fx'$. Note
first that commutativity in~\eqref{lift} forces $f^\ast(H)(x) = H(fx)$ and
$f^\ast(H)(x') = H(fx')$ for every $H \in \cat{Set}^{\cat D}$. In particular,
taking $H = y_{fx}$, we obtain a map of sets
\begin{equation*}
    f^\ast(y_{fx})(g) \colon \cat D(fx, fx) \to \cat D(fx, fx')
\end{equation*}
and evaluating this at $1_{fx}$ yields the required morphism $fx \to fx'$.
Straightforward diagram chasing shows this assignation to be functorial, and
that the resultant functor uniquely induces the lifting $f^\ast$. Thus we have
a full embedding $A \colon \cat{Cat} \to \cat{BIon}$; in
Remark~\ref{alexspecionad} below, we will see that this embedding is in fact
coreflective.\end{Ex}
\begin{Rk}\label{alternatemap}
To give a continuous map of ionads $(X, I) \to (Y, J)$ is equally well to give
a function $f \colon X \to Y$ and a natural transformation $\delta \colon
f^{-1} J \Rightarrow I f^{-1}$ such that the diagrams
\begin{equation}\label{twoaxioms}
    \cd[@+0.3em]{
    f^{-1} J \ar[r]^-{f^{-1} \epsilon} \ar[d]_{\delta} & f^{-1} \\
    I f^{-1} \ar[ur]_{\epsilon f^{-1}}} \qquad \text{and} \qquad
    \cd[@C+0.5em]{
    f^{-1} J \ar[r]^-{f^{-1} \Delta} \ar[d]_{\delta} & f^{-1} JJ \ar[r]^{\delta J} & I f^{-1} J \ar[d]^{I \delta} \\
    I f^{-1} \ar[rr]_{\Delta f^{-1} } & & II f^{-1}}
\end{equation}
commute; that is, such that the pair $(f^{-1}, \delta)$ is a comonad
morphism$(\cat{Set}^Y, J) \to (\cat{Set}^X, I)$ in the sense
of~\cite{Street1972formal}. The passage between the two descriptions is as
follows: given $\delta \colon f^{-1} J \Rightarrow I f^{-1}$, we define the
corresponding $f^\ast \colon \cat O(Y) \Rightarrow \cat O(X)$ by
$f^\ast(a\colon A \to JA) = \delta_A.f^{-1}a \colon f^{-1}A \to If^{-1}A$.
Conversely, given $f^\ast$, we obtain the $A$-component of the corresponding
$\delta$ by applying $f^\ast$ to the cofree $J$-coalgebra $\Delta_A \colon JA
\to JJA$---yielding an $I$-coalgebra $f^\ast(\Delta_A) \colon f^{-1}JA \to
If^{-1}JA$---and then postcomposing with $If^{-1}\epsilon_A \colon If^{-1}JA
\to If^{-1}A$.
\end{Rk}
\begin{Ex}\label{basisex}
Let $Y$ be the ionad generated by a basis $M \colon \cat B \to \cat{Set}^Y$. By
the preceding Remark, to give a continuous map $X \to Y$ is to give a function
$f \colon X \to Y$ together with a natural transformation $f^{-1} . M \otimes
[M, \thg] \Rightarrow I_X . f^{-1}$ satisfying the axioms of~\eqref{twoaxioms}.
Now by virtue of the adjunction $M\otimes (\thg) \dashv [M, \thg]$, to give
this natural transformation is equally well to give a natural transformation
$f^{-1} . M \otimes (\thg) \Rightarrow I_X . f^{-1} . M\otimes (\thg) \colon
\cat{Set}^{\cat B^\op} \to \cat{Set}^X$; and since both $f^{-1}$ and
$I_X.f^{-1}$ preserve colimits, such a natural transformation is uniquely
determined by a natural transformation $\alpha \colon f^{-1}.M \Rightarrow
I_X.f^{-1}.M \colon \cat B \to \cat{Set}^X$. Under these correspondences, the
two axioms of~\eqref{twoaxioms} become two axioms on $\alpha$ which say exactly
that it equips $f^{-1}.M$ with the structure of a coalgebra for the comonad
$(I_X)^{\cat B}$ on $(\cat{Set}^X)^{\cat B}$. But to give such a coalgebra
structure on $f^{-1}.M$ is equally well to give a lifting
%
%
%
%
% which in turn is equally well
%to give a lifting
%\begin{equation}\label{lifting1}
%  \cd{ \cat{Set}^{\cat B^\op} \ar[d]_{M \otimes (\thg) } \ar@{-->}[r] & \cat O(X) \ar[d]^{U_X} \\
%	\cat{Set}^Y \ar[r]_{f^{-1}}  & \cat{Set}^X
%  }
%\end{equation}
\begin{equation}\label{lifting2}
  \cd{ \cat B \ar[d]_{M} \ar[r]^{f'} & \cat O(X) \ar[d]^{U_X} \\
	\cat{Set}^Y \ar[r]_-{f^{-1}} & \cat{Set}^X
  }
\end{equation}
of $f^{-1}.M$ through $\cat O(X)$, the category of $I_X$-coalgebras.

A consequence of this is that the collection of ionad morphisms $X \to Y$ will
be a set whenever $Y$ is bounded. For in this case, continuous maps are given
by diagrams as in~\eqref{lifting2}; there are only a set of functions $f \colon
X \to Y$; and for every such such $f$, there are only a set of liftings $f'$,
since $\cat B$ is small and each $H \in \cat{Set}^X$ admits only a set of
$I_X$-coalgebra structures. In particular, we deduce that the category
$\cat{BIon}$ of bounded ionads is locally small.
%We now
%claim that such liftings are in natural bijection with liftings
%as indicated. Certainly we may extend a lifting as
%in~\eqref{lifting2} to a lifting as in~\eqref{lifting1}, using
%the fact that every object in $\cat{Set}^{\cat B^\op}$ is a
%canonical colimit of representables, and that such colimits are
%preserved strictly by $f^{-1} . M \otimes (\thg)$ and created
%strictly by $U_X$. Conversely, given a lifting as
%in~\eqref{lifting1}, restricting along the Yoneda embedding
%yields a square as in~\eqref{lifting2}, but which may commute
%only up to isomorphism. But by Remark~\ref{pseudo}, we may
%replace this square with another which is genuinely
%commutative. It's now easy to verify that these two processes
%are inverse to each other.
\end{Ex}
\begin{Rk}\label{counit}
Writing $\cat{Set}^{(\thg)} \colon \cat{Set} \to \cat{CAT}^\op$ for the functor
sending $X$ to $\cat{Set}^X$ and $f \colon X \to Y$ to $f^{-1} \colon
\cat{Set}^Y \to \cat{Set}^X$, we may regard $\cat{Ion}$ as a full subcategory
of the comma category $\cat{CAT}^\op \downarrow \cat{Set}^{(\thg)}$. The
preceding example shows that, for any basis $M \colon \cat B \to \cat{Set}^Y$,
the corresponding ionad $Y$ is a coreflection of $M$ into this full
subcategory. The counit of this coreflection is a map
\begin{equation*}
  \cd[@!C]{\cat B \ar[rr]^{\overline M} \ar[dr]_M & & \cat O(Y) \ar[dl]^{U_Y} \\ & \cat{Set}^Y}
\end{equation*}
composition with which induces the bijection between squares of the
form~\eqref{mapdef} and of the form~\eqref{lifting2}.
\end{Rk}
\begin{Ex}\label{topembed}
Let $f \colon X \to Y$ be a continuous map of topological spaces. We have a
commutative diagram
\begin{equation*}
\cd{
    \O(Y) \ar[r]^{f^\ast} \ar@{ >->}[d] & \O(X) \ar@{ >->}[d] \\
    \P Y \ar@{ >->}[d] & \P X \ar@{ >->}[d] \\
    \cat{Set}^Y \ar[r]_{f^{-1}} & \cat{Set}^X
}
\end{equation*}
and so applying the coreflection of the preceding Remark, we obtain a
continuous map of ionads $\Sigma X \to \Sigma Y$. Thus the assignation $X
\mapsto \Sigma X$ extends to a functor $\Sigma \colon \cat{Sp} \to \cat{Ion}$,
which in fact exhibits $\cat{Sp}$ as a full reflective subcategory of
$\cat{Ion}$ (and indeed also of $\cat{BIon}$). To see this, we first exhibit a
left adjoint $\Lambda$ for $\Sigma$. Given an ionad $X$, we observe that $\P X$
is the lattice of subobjects of $1$ inside $\cat{Set}^X$; and that since $I_X$
is cartesian, it restricts and corestricts to this lattice, thus yielding an
interior operator $i \colon \P X \to \P X$. We claim that the space $\Lambda X$
with this interior operator gives a reflection of $X$ along the functor
$\Sigma$. To see this, first observe that the open sets of $\Lambda X$, which
are the fixpoints of~$i$, are precisely the subobjects of $1$ inside $\cat
O(X)$; and so we have a pullback diagram
\begin{equation}\label{pbdiag}
 \cd{\O(\Lambda X) \ar@{ >->}[r] \ar@{ >->}[d] & \cat O(X) \ar[d] \\ \P X \ar@{ >->}[r] & \cat{Set}^X\rlap{ .}}
\end{equation}
Now if $Y$ is a topological space, then by Example~\ref{basisex}, to give an
ionad map $X \to \Sigma Y$ is to give a function $f \colon X \to Y$ and a
functor $f' \colon \O(Y) \to \cat O(X)$ making the following square commute:
\begin{equation*}
\cd{
    \O(Y) \ar[r]^{f'} \ar[d]_M & \cat O(X) \ar[d]^{U_X} \\
    \cat{Set}^Y \ar[r]_{f^{-1}} & \cat{Set}^X\rlap{ .}
}
\end{equation*}
But the lower composite $f^{-1} . M$ factors through the inclusion $\P X
\rightarrowtail \cat{Set}^X$ (since $f^{-1}$ preserves finite limits, and hence
subobjects of $1$); and so, since~\eqref{pbdiag} is a pullback, $f'$ must
factor through the inclusion $\O(\Lambda X) \rightarrowtail \cat O(X)$. But
this implies that there is at most one $f'$ lifting $f^{-1}$, and that such a
lifting will exist precisely when $f$ is continuous as a map of spaces $\Lambda
X \to Y$. Thus for each $X$, we have established a bijection $\cat{Ion}(X,
\Sigma Y) \cong \cat{Sp}(\Lambda X, Y)$---whose naturality in $Y$ is easily
checked---so that the assignation $X \mapsto \Lambda X$ extends to a functor
$\Lambda \colon \cat{Ion} \to \cat{Sp}$ left adjoint to $\Sigma$. It remains to
observe that for any space $Z$, we have $\Lambda \Sigma Z = Z$, so that the
adjunction $\Lambda \dashv \Sigma$ is a reflection as required.
\end{Ex}

\begin{Rk}
The reflection constructed in the previous example again has a familiar
topos-theoretic interpretation. Given an ionad $X$, we may factorise the unique
geometric morphism $\cat{O}(X) \to \cat{Set}$ (note that this exists since
$\cat O(X)$ is cocomplete) as
\begin{equation*}
    \cat{O}(X) \xrightarrow{\ i\ } \E \xrightarrow{\ p\ } \cat{Set}
\end{equation*}
where $i$ is hyperconnected and $p$ is localic (cf.~\cite[\S
A4.6]{Johnstone2002Sketches}). That $p$ is localic means that $\E$ is
equivalent to $\cat{Sh}(K)$ for some locale $K$; whilst that $i$ is
hyperconnected means in particular that it is a surjection. Hence so also is
the composite geometric morphism
\begin{equation}\label{compos}
    \cat{Set}^X \xrightarrow{\ \ } \cat{O}(X) \xrightarrow{\ i\ } \E\ \text.
\end{equation}
But since $\cat{Set}^X$ is itself equivalent to the category of sheaves on the
discrete locale $\P X$, the geometric morphism~\eqref{compos} is induced by a
surjective locale morphism $\P X \to K$; and, as we noted in
Example~\ref{discrete}, to give this is to give a topological space with set of
points $X$: which is the reflection of $X$ into $\cat{Sp}$ described above.
\end{Rk}

%\begin{Rk}\label{locsmall}
%
%This is very similar to what happens in a predicative type theory, or in a
%category with small maps in the sense of algebraic set
%theory~\cite{Joyal1995Algebraic}: and it seems that the existence of
%classifying toposes, and the non-existence of classifying ionads, reflects a
%fundamental impredicativity in the former theory (where large categories can
%classify other large categories) which is not present in the latter (where
%large categories can classify only small ones).
%\end{Rk}

\section{The specialisation enrichment}\label{s5}
Recall that if $X$ is a space, then its set of points may be preordered by the
\emph{specialisation order}, in which $x \leqslant y$ whenever every open set
of $X$ that contains $x$ also contains $y$. The specialisation order induces a
preordering on each hom-set $\cat{Sp}(X,Y)$ in which $f \leqslant g$ iff $fx
\leqslant gx$ for all $x \in X$, and with respect to these preorderings, the
composition functions $\cat{Sp}(Y, Z) \times \cat{Sp}(X, Y) \to \cat{Sp}(X,Z)$
become order-preserving maps. Consequently, this structure enriches $\cat{Sp}$
to a locally preordered $2$-category. We now wish to describe a corresponding
enrichment of $\cat{Ion}$ to a (no longer locally preordered) $2$-category. In
order to do so, we first recast the definition of the $2$-cells of $\cat{Sp}$
in a manner which makes the correct generalisation obvious. Observe that $f
\leqslant g \colon X \to Y$ just when $fx \in U$ implies $gx \in U$ for every
$x \in X$ and open $U \subseteq Y$. This is equivalently to ask that $f^{-1}(U)
\subseteq g^{-1}(U)$ for every open set $U \subseteq Y$: or in other words,
that there should be an inequality
\begin{equation*}
    \cd{
        \O(Y) \ar@/^1.2em/@{-->}[r]^{f^\ast} \ar@/_1.2em/@{-->}[r]_{g^\ast} \ar@{}[r]|{\leqslant}  &  \O(X) \\
   }
\end{equation*}
between the (unique) liftings of $f^{-1}$ and $g^{-1}$ through the
corresponding open set lattices. We are therefore led to propose:
\begin{Defn}\label{specdef}
A \emph{specialisation} between ionad morphisms $f, g \colon X \to Y$ is a
natural transformation
\begin{equation}\label{spec2cell}
    \cd{
        \cat O(Y) \ar@/^1.2em/[r]^{f^\ast} \ar@/_1.2em/[r]_{g^\ast} \dtwocell{r}{\alpha} &  \cat O(X)
   }
\end{equation}
subject to no further conditions (in particular, this means no compatibility
conditions with $f^{-1}$ or $g^{-1}$). The ionads, continuous maps and
specialisations form a $2$-category, for which we reuse the notation
$\cat{Ion}$; similarly, we write $\cat{BIon}$ to denote the full and locally
full sub-$2$-category spanned by the bounded ionads.
\end{Defn}

%\begin{Rk}
%This definition introduces some ambiguity, since when we speak of $\cat{Ion}$
%or $\cat{BIon}$, the context may not be enough to determine whether it is the
%$2$-category, or the underlying $1$-category, which is being referred to. To
%resolve this we adopt the notational convention of~\cite{Kelly1982Basic}: thus
%$\cat{Ion}$ and $\cat{BIon}$ will \emph{always} denote $2$-categories, with the
%underlying $1$-categories being denoted instead by $\cat{Ion}_0$ and
%$\cat{BIon}_0$. Since most of the other categories we have considered so far
%also possess natural two-dimensional structure, we extend this convention to
%them, writing $\cat{Cat}$, $\cat{Top}$, $\cat{GTop}$ or $\cat{Sp}$ now to
%denote the $2$-category, with $\cat{Cat}_0$, $\cat{Top}_0$, $\cat{GTop}_0$ or
%$\cat{Sp}_0$ the underlying $1$-category.
%\end{Rk}
%

\begin{Rk}
We saw in Remark~\ref{ctstopos} that the assignation $(X, I) \mapsto \cat O(X)$
yields a functor $\cat O(\thg) \colon \cat{Ion} \to \cat{Top}$, with a
continuous map $(f, f^\ast)$ of ionads being sent to the geometric morphism
whose inverse image part is $f^\ast$. Consequently, the specialisations between
continuous maps are in bijection with the geometric transformations between the
corresponding geometric morphisms, so that the functor $\cat O(\thg)$ extends
to a locally fully faithful $2$-functor.
\end{Rk}

\begin{Ex}
If $F, G \colon \cat C \to \cat D$ are functors between small categories, then
the specialisations $A(F) \Rightarrow A(G) \colon A(\cat C) \Rightarrow A(\cat
D)$ are given by natural transformations $\cat{Set}^F \Rightarrow \cat{Set}^G
\colon \cat{Set}^{\cat D} \to \cat{Set}^{\cat C}$; and these are in bijection
with natural transformations $F \Rightarrow G \colon \cat C \to \cat D$. Thus
the ordinary functor $A \colon \cat{Cat} \to \cat{BIon}$ extends to a $2$-fully
faithful $2$-functor.
\end{Ex}

\begin{Rk}\label{justify}
%Even though the definition of specialisation
%is obtained by direct analogy with the topological case, it may still appear a
%little ad hoc. A further justification for it arises from the
%paper~\cite{Lack2002Formal}, which studies the construction $\cat{EM}(\thg)$
%that freely completes a $2$-category under Eilbenberg-Moore objects for monads.
%It follows from Section 2.1 of that paper that the $2$-category
%$\cat{EM}(\cat{CAT}^\co)^\co$---which is the free completion of $\cat{CAT}$
%under Eilenberg-Moore objects for \emph{co}monads---may be described as
%follows. Its objects are pairs $(\cat C, I)$ where $\cat C$ is a category and
%$I$ a comonad on it; its $1$-cells $(\cat C, I) \to (\cat D, J)$ are given by a
%functor $f \colon \cat C \to \cat D$ together with a lifting $\bar f \colon
%I\text-\cat{Coalg} \to J\text-\cat{Coalg}$ through the corresponding categories
%of coalgebras; and its $2$-cells $(f, \bar f) \Rightarrow (g, \bar g)$ are
%natural transformations $\alpha \colon \bar f \Rightarrow \bar g$ subject to no
%further conditions. Thus $\cat{Ion}$ is nothing more than a particular
%sub-$2$-category of $\cat{EM}(\cat{CAT}^\co)^{\co\op}$.
%
%This alternative justification for the $2$-cells of $\cat{Ion}$ also furnishes
%us with an alternative description of them.
Recall from Remark~\ref{alternatemap} that a map of ionads $(X, I) \to (Y, J)$
corresponds to a pair $(f, \delta)$ where $f \colon X \to Y$ and $\delta \colon
f^{-1} J \Rightarrow I f^{-1}$ is a natural transformation satisfying two
axioms. If $(f, \delta)$ and $(g, \gamma)$ are a parallel pair of ionad
morphisms given in this manner, then the specialisations between them are in
correspondence with natural transformations $\rho \colon f^{-1} J \Rightarrow
g^{-1}$ making the diagram
\begin{equation}\label{2celldiag}
    \cd[@C+1em]{
      f^{-1} J \ar[r]^-{f^{-1} \Delta} \ar[d]_{f^{-1} \Delta} &
      f^{-1} JJ \ar[r]^-{\rho J} &
      g^{-1} J \ar[d]^{\gamma} \\
      f^{-1} JJ \ar[r]_-{\delta J} & I f^{-1} J \ar[r]_-{I \rho} & I g^{-1}
    }
\end{equation}
commute. This is a consequence of~\cite[Section~2.1]{Lack2002Formal}; we
summarise the argument as follows. Given a natural transformation $\alpha
\colon f^\ast \Rightarrow g^\ast \colon \cat O(Y) \to \cat O(X)$, then the
component at $A$ of the corresponding $\rho \colon f^{-1} J \Rightarrow g^{-1}$
is obtained by first evaluating $U_X. \alpha$ at the cofree coalgebra $\Delta_A
\colon JA \to JJA$---which yields a map $f^{-1} JA \to g^{-1} JA$---and then
postcomposing this with $g^{-1} \epsilon_A \colon g^{-1} JA \to g^{-1} A$.
Conversely, if given $\rho \colon f^{-1} J \Rightarrow g^{-1}$
making~\eqref{2celldiag} commute, then the corresponding $\alpha \colon f^\ast
\Rightarrow g^\ast$ has its component at a coalgebra $a \colon A \to JA$ given
by $\rho_A .f^{-1} a \colon f^\ast A \to g^\ast A$.
\end{Rk}

\begin{Ex}\label{basisex2} Recall from Example~\ref{basisex} that if the ionad $Y$ is
generated by the basis $M \colon \cat B \to \cat{Set}^Y$, then ionad morphisms
$X \to Y$ are in bijection with pairs $(f, f')$ where $f \colon X \to Y$ is a
function and $f' \colon \cat B \to \cat O(X)$ a functor making~\eqref{lifting2}
commute. Suppose now that $(f, f')$ and $(g,g')$ are two ionad morphisms
specified in this way. By the preceding Remark, the specialisations between
them correspond with natural transformations $\rho \colon f^{-1} J \Rightarrow
g^{-1}$ making~\eqref{2celldiag} commute, where $J$ is the comonad $M \otimes
[M, \thg]$ generated by the basis $M$. But to give such a $\rho$ is equally
well to give a natural transformation $f^{-1} . M \otimes (\thg) \Rightarrow
g^{-1} . M \otimes (\thg)$ satisfying one axiom; and since both $f^{-1}$ and
$g^{-1}$ preserve colimits, such a natural transformation is determined
uniquely by a natural transformation $\phi \colon f^{-1}.M \Rightarrow
g^{-1}.M$ satisfying one axiom, which is easily shown to amount to the
requirement that $\phi$ should lift
%
% To see what this
%axiom is, recall that the functors $f', g' \colon \cat B \to \cat O(X)$
%correspond to natural transformations $\sigma \colon f^{-1}.M \Rightarrow
%I_X.f^{-1}.M$ and $\tau \colon g^{-1}.M \Rightarrow I_X.g^{-1}.M$ satisfying
%two coalgebra axioms; and in these terms, the axiom that $\phi$ must satisfy
%requires commutativity in
%\begin{equation*}
%    \cd{
%        f^{-1}.M \ar[r]^{\phi} \ar[d]_{\sigma} &
%        g^{-1}.M \ar[d]^{\tau} \\
%        I_X.f^{-1}.M \ar[r]_{I_X.\phi} &
%        I_X.g^{-1}.M\rlap{ ;}
%    }
%\end{equation*}
%which is precisely to ask for a lifting of the $2$-cell $\phi$
through the
forgetful functor $U_X \colon \cat O(X) \to \cat{Set}^X$. Thus we have shown
that to give a specialisation from $(f, f')$ to $(g, g')$ is equally well to
give a natural transformation $\alpha \colon f' \Rightarrow g' \colon \cat{B}
\to \cat{O}(X)$.

We may deduce from this that the category $\cat{Ion}(X, Y)$ is small whenever
$Y$ is a bounded ionad. Indeed, we know from Example~\ref{basisex} that this
category has only a set of objects; moreover, the collection of morphisms
between any two such objects may be identified with the collection of natural
transformations $f' \Rightarrow g'$ for some $f', g' \colon \cat B \to \cat
O(X)$, and this is a set because $\cat B$ is small. In particular, we may
conclude that $\cat{BIon}$ is a locally small $2$-category.
\end{Ex}
\begin{Ex}
Taking the preceding example together with Example~\ref{topembed}, we see that
if $f, g \colon X \to Y$ are continuous maps of topological spaces, then the
specialisations $\Sigma f \Rightarrow \Sigma g \colon \Sigma X \to \Sigma Y$
are given by natural transformations
\begin{equation*}
    \cd[@R-2em]{ & \O(X) \ar@{ >->}[dr] \\
        \O(Y) \ar[ur]^{f^{\ast}} \ar[dr]_{g^{\ast}}  \dtwocell{rr}{\alpha} & & \cat O(\Sigma X)\\
        & \O(X) \ar@{ >->}[ur]}
\end{equation*}
as indicated. But as the embedding $\O(X) \rightarrowtail \cat O(\Sigma X)$ is
full and faithful, any such natural transformation is induced by a unique
natural transformation $f^\ast \Rightarrow g^\ast$; and as $\O(X)$ is a poset,
there can be at most one such, which exists precisely when $f \leqslant g$ in
the $2$-category $\cat{Sp}$. Thus we have shown that the embedding functor
$\Sigma \colon \cat{Sp} \to \cat{BIon}$ extends to a $2$-fully faithful
$2$-functor; and much the same argument shows that the left adjoint of $\Sigma$
extends to a left $2$-adjoint.
\end{Ex}
When at the start of this Section we described the two-dimensional structure of
$\cat{Sp}$, we did so by constructing it from the specialisation order on each
space. By contrast, when we defined the two-dimensional structure of
$\cat{Ion}$ we did so directly; and this raises the question of what the
appropriate ionad-theoretic analogue of the specialisation order should be. In
order to answer this, let us observe that in $\cat{Sp}$, the specialisation
order on a space $X$ is encoded in the two-dimensional structure as the
hom-category $\cat{Sp}(1, X)$. This immediately suggests the following:
\begin{Defn}\label{specfunctor}
The \emph{specialisation functor} $V \colon \cat{BIon} \to \cat{Cat}$ is the
representable functor $\cat{BIon}(1, \thg)$, and its value at a bounded ionad
$X$ is the \emph{specialisation category} of $X$.
\end{Defn}
Observe that we are forced to define $V$ only on the bounded ionads, since if
$X$ is an unbounded ionad, then there is no reason to expect that
$\cat{BIon}(1, X)$ should be a small category (though it will always have a
mere set of objects).
\begin{Rk}\label{alexspec}
In the topological case, the $2$-functor $\cat{Sp}(1, \thg) \colon \cat{Sp} \to
\cat{Poset}$ is right $2$-adjoint to the $2$-functor $\cat{Poset} \to \cat{Sp}$
sending a poset to the corresponding Alexandroff space. We shall see in
Remark~\ref{alexspecionad} below that the corresponding result holds for
ionads: the $2$-functor $V \colon \cat{BIon} \to \cat{Cat}$ of the previous
definition is right $2$-adjoint to the embedding $A \colon \cat{Cat} \to
\cat{BIon}$.
\end{Rk}

\begin{Rk}\label{vx}
We may extract the following explicit description of the specialisation
category $VX$ of a bounded ionad $X$. An object of $VX$ is given by a function
$1 \to X$, which is a point $x \in X$, together with a lifting $x^\ast$ making
\begin{equation*}
    \cd{
        \cat O(X) \ar[r]^-{x^\ast} \ar[d]_{U_X} & \cat{Set} \ar[d]^{\id} \\
        \cat{Set}^X \ar[r]_-{x^{-1}} & \cat{Set}
   }
\end{equation*}
commute. Obviously, there is exactly one such lifting, namely $x^\ast =
x^{-1}.U_X$, and so we may identify objects of $VX$ with elements of $X$. Now
the morphisms $x \to y$ in $VX$ are the specialisations $(x, x^\ast)
\Rightarrow (y, y^\ast)$, and these are given by natural transformations
$x^{-1}.U_X \Rightarrow y^{-1}.U_X \colon \cat O(X) \to \cat{Set}$.

A priori this is as much as we can say; however, if we suppose given some basis
$M \colon \cat B \to \cat{Set}^X$ which generates the ionad $X$ then we can
simplify this description further. For then, by Remark~\ref{basisex2}, we may
identify specialisations $(x, x^\ast) \Rightarrow (y, y^\ast)$ with natural
transformations $x^{-1}.U_X.\overline M \Rightarrow y^{-1}.U_X.\overline M$,
where $\overline M \colon \cat B \to \cat O(X)$ is the coreflection map of
Remark~\ref{counit}; and since $x^{-1}.U_X.\overline M = x^{-1} . M =
M(\thg)(x)$ and likewise $y^{-1}.U_X.\overline M = M(\thg)(y)$, we may identify
these in turn with natural transformations $M(\thg)(x) \Rightarrow M(\thg)(y)
\colon \cat B \to \cat{Set}$. Thus we arrive at the following simple
description of $VX$. Starting from the basis $M \colon \cat B \to \cat{Set}^X$
we may transpose it to a functor $M' \colon X \to \cat{Set}^{\cat B}$, and $VX$
is now obtained by factorising $M'$ as a functor bijective on objects, followed
by one that is fully faithful.
%\begin{equation*}
%    M' = X \xrightarrow{u} VX \xrightarrow{v} \cat{Set}^{\cat B}
%\end{equation*}
%where $u$ is bijective on objects, and $v$ is fully faithful.
\end{Rk}
\section{Limits and colimits of ionads}\label{lims}
In this section, we describe, with sketches of proofs, the limits and colimits
that exist in the $2$-category of ionads. It turns out that in the category of
\emph{all} ionads, rather few of these exist:
\begin{Prop}\label{ioncolims}
The $2$-category $\cat{Ion}$ has coproducts, tensors by small categories, and a
terminal object.
\end{Prop}
\begin{proof}
The terminal object of $\cat{Ion}$ is given by $(1, \id_{\cat{Set}})$. Given a
family $(X_k, I_k)_{k \in K}$ of ionads, their coproduct has as its underlying
set $\sum_{k} X_k$ and as its interior comonad the composite
\begin{equation*}\textstyle
    \cat{Set}^{\sum_{k} X_k} \xrightarrow{\ \cong\ } \prod_{k} \cat{Set}^{X_k} \xrightarrow{\ \prod_{k} I_k\ }
    \prod_{k} \cat{Set}^{X_k} \xrightarrow{\ \cong\ }\cat{Set}^{\sum_{k} X_k}\ \text.
\end{equation*}
As regards tensor products, suppose given an ionad $X$ and a small category
$\cat C$. We define the ionad $\cat C \otimes X$ to have underlying set $\ob
\cat C \times X$, and interior comonad generated by the composite adjunction
\begin{equation}\label{adjtensor}
    \cd[@C+2.5em]{
      \cat{Set}^{\ob \cat C \times X} \cong (\cat{Set}^{X})^{\ob \cat C}
        \ar@<-4pt>[r]_-{(\text{cofree})^{\ob \cat C}}
        \ar@{}[r]|-{\bot} &
      {\cat{O}(X)^{\ob \cat C}}
        \ar@<-4pt>[r]_-{\mathrm{Ran}_J}
        \ar@{}[r]|-{\bot}
        \ar@<-4pt>[l]_-{(\text{forget})^{\ob \cat C}} &
      {\cat{O}(X)^{\cat C}}
        \ar@<-4pt>[l]_-{\cat{O}(X)^J}\rlap{ ,}
      }
\end{equation}
where $J \colon \ob \cat C \to \cat C$ is the canonical inclusion. Observe that
in order for $\mathrm{Ran}_J$ to exist here we must know $\cat O(X)$ to be
complete: but being a topos, it is complete if and only if cocomplete, and it
is certainly the latter by virtue of being comonadic over $\cat{Set}^X$. Since
both left adjoint functors in~\eqref{adjtensor} strictly create equalisers, the
adjunction they generate is strictly comonadic, so that $\cat O(\cat C \otimes
X) \cong \cat O(X)^{\cat C}$: from which the universal property of the tensor
product follows easily.
\end{proof}

However, on restricting to the $2$-category of bounded ionads, the situation is
much more satisfying:

\begin{Thm}\label{bioncocomp}
$\cat{BIon}$ is cocomplete as a $2$-category.
\end{Thm}
\begin{proof}[Proof (sketch)]
It is easy to see that the constructions of coproducts and tensor products in
$\cat{Ion}$ restrict to $\cat{BIon}$.
%: indeed, given a family $M_i
%\colon \cat B_i \to \cat{Set}^{X_i}$ of bases, the coproduct of the
%corresponding ionads is generated by the basis
%\begin{align*}\textstyle
%N \colon \sum_i \cat B_i & \to \cat{Set}^{\sum_i X_i}\\
%N(i,b)(j,x) & = \begin{cases} M_i(b)(x) & \text{if $i = j$;}\\ 0 & \text{otherwise.}
%\end{cases}\end{align*}
It therefore suffices to prove that $\cat{BIon}$ has coequalisers. Given a
parallel pair
 $f, g \colon X \rightrightarrows Y$, we first form the coequaliser $q \colon Y \to Z$ of the
functions between the underlying sets of points, and then the equaliser $E
\colon \cat{E} \to \cat{O}(Y)$ of $f^\ast, g^\ast \colon \cat O (Y)
\rightrightarrows \cat O (X)$ in $\cat{CAT}$. Observing that $q^{-1}$ is the
equaliser of $f^{-1}$ and $g^{-1}$, we thereby induce a morphism $V \colon \cat
E \to \cat{Set}^Z$ fitting into a commutative diagram:
\begin{equation*}
  \cd{
     \cat{E} \ar[r]^-{E} \ar[d]_{V} & \cat{O}(Y) \ar@<3pt>[r]^{f^\ast} \ar@<-3pt>[r]_{g^\ast} \ar[d]_{U_Y} & \cat{O}(X) \ar[d]^{U_X} \\
     \cat{Set}^Z \ar[r]_{q^{-1}} & \cat{Set}^Y \ar@<3pt>[r]^{f^{-1}} \ar@<-3pt>[r]_{g^{-1}} & \cat{Set}^X\rlap{ .}
  }
\end{equation*}
We shall show that $\cat E$ is isomorphic to the category of opens of a bounded
ionad structure on the set $Z$; it is then easy to see that this ionad must be
the coequaliser of $f$ and $g$ in $\cat{BIon}$. The key step in the proof will
be to show that $\cat E$ is an accessible category, which we will do using the
fact that the $2$-category $\cat{ACC}$ of accessible categories is closed under
the formation of inserters and equifiers in $\cat{CAT}$ (see ~\cite[Theorem
5.1.6]{Makkai1989Accessible}). So let $\gamma \colon f^{-1} . J \Rightarrow I .
f^{-1}$ and $\delta \colon g^{-1} . J \Rightarrow I . g^{-1}$ be the natural
transformations corresponding to the functors $f^\ast$ and $g^\ast$. Now an an
object of $\cat E$ consists of a coalgebra $a \colon A \to JA$ in $\cat{Set}^Y$
such that the equality
\begin{equation*}
  f^{-1} A \xrightarrow{f^{-1} a} f^{-1} JA \xrightarrow{\gamma_A} I f^{-1} A
\quad = \quad
  g^{-1} A \xrightarrow{g^{-1} a} g^{-1} JA \xrightarrow{\delta_A} I g^{-1} A
\end{equation*}
holds. In particular, this means that $f^{-1} A = g^{-1} A$, so that to give
such an object is equally well to give an object $W \in \cat{Set}^Z$ and a
coalgebra $a \colon q^{-1} W \to J q^{-1} W$ such that $\gamma_X . f^{-1}
q^{-1} u = \delta_X . g^{-1} q^{-1} u$. Using this explicit description of
$\cat E$, it is easy to give a construction of it from inserters and equifiers
in $\cat{ACC}$: first we form the inserter of the two functors $q^{-1}, Jq^{-1}
\colon \cat{Set}^Z \rightrightarrows \cat{Set}^Y$, and then equify three pairs
of $2$-cells, imposing the coalgebra axioms and the additional compatibility
with $\gamma_X$ and $\delta_X$. It follows that $\cat E$ is accessible as
claimed.
%\begin{equation*}
%  \cd{
%     \cat{Set}^Z \ar@<3pt>[r]^{q^{-1}} \ar@<-3pt>[r]_{J q^{-1}} & \cat{Set}^Y
%  }\ \text.
%\end{equation*}
%Thus we have a functor $H \colon \cat C \to \cat{Set}^Z$ and a universal
%$2$-cell $\theta \colon q^{-1} H \Rightarrow J q^{-1} H$. Next we equify, in
%succession, the following three pairs of $2$-cells:
%\begin{equation*}
%  \cd[@C+6em]{
%     \cat{C} \ar@/^1.5em/[r]^{q^{-1} H} \ar@/_1.5em/[r]_{q^{-1} H} \dtwocell[0.35]{r}{\id} \dtwocell[0.55]{r}{\epsilon q^{-1} H . \theta} & \cat{Set}^Y
%  } \qquad \text{,} \qquad
%  \cd[@C+6em]{
%     \cat{C} \ar@/^1.5em/[r]^{q^{-1} H} \ar@/_1.5em/[r]_{JJq^{-1} H} \dtwocell[0.3]{r}{J \theta . \theta} \dtwocell[0.55]{r}{\Delta q^{-1} H . \theta} & \cat{Set}^Y
%  }
%\end{equation*}
%\begin{equation*}
%\text{and} \qquad
%  \cd[@C+11em]{
%     \cat{C} \ar@/^1.5em/[r]^{f^{-1} q^{-1} H} \ar@/_1.5em/[r]_{I f^{-1} q^{-1} H} \dtwocell[0.24]{r}{\gamma q^{-1} H . f^{-1} \theta} \dtwocell[0.58]{r}{\delta q^{-1} H . g^{-1} \theta} & \cat{Set}^Y
%  }\ \text.
%\end{equation*}
%It's easy to see that the resultant category is isomorphic to $\cat E$ and so
%$\cat E$ is accessible as claimed.

We may now show by a straightforward diagram chase that, because $U_Y$ and
$U_X$ strictly create colimits and finite limits, so too does $V$. Since
$\cat{Set}^Z$ has all colimits, it follows that $\cat E$ does too, and that $V$
preserves them. Hence $\cat E$ is locally presentable, and by the special
adjoint functor theorem, $V$ has a right adjoint $G$. Moreover, since $V$
creates finite limits, it in particular preserves them, so that the composite
$VG$ is the interior comonad of an ionad on $Z$. Since $V$ strictly creates
equalisers, it is strictly comonadic, so that $\cat E$ is isomorphic to the
category of opens of this ionad; and it remains only to show the ionad to be
bounded. But since $\cat E$ is locally presentable, it is in particular a
Grothendieck topos, so we are done by Proposition~\ref{boundprop}.
\end{proof}

We have similarly good behaviour with respect to limits:

\begin{Thm}\label{bioncomplete}
$\cat{BIon}$ is complete as a $2$-category.
\end{Thm}
\begin{proof}[Proof (sketch)]
We begin by showing that $\cat{BIon}$ is complete as a $1$-category. First we
prove that the forgetful functor $U \colon \cat{BIon} \to \cat{Set}$ is a
fibration; then we show that all the fibres of $U$ are complete; and then we
show that reindexing between those fibres preserves limits. These three
conditions together imply that $\cat{BIon}$, the total category of this
fibration, is complete as a $1$-category.

To show that $U \colon \cat{BIon} \to \cat{Set}$ is a fibration, we must, given
a bounded ionad $(Y, J)$ and a map of sets $f \colon X \to Y$, produce a
cartesian lift $(f, f^\ast) \colon (X, I) \to (Y, J)$ in $\cat{BIon}$. We take
$I \colon \cat{Set}^X \to \cat{Set}^X$ to be the comonad generated by the
string of adjunctions
\begin{equation*}
    \cd{\cat{Set}^{X} \ar@<-4pt>[r]_{\Pi_f} \ar@{}[r]|{\bot} & \cat{Set}^{Y} \ar@<-4pt>[r]_{\text{cofree}} \ar@{}[r]|{\bot} \ar@<-4pt>[l]_{f^{-1}} & {\cat{O}(Y)} \ar@<-4pt>[l]_{\text{forget}}}\ \text,
\end{equation*}
% Since each
%functor in this diagram is cartesian, so is $I$. For accessibility, since $I =
%f^{-1} . J . \Pi_f$ it suffices to show that each of these three factors is
%accessible. This is clear for $f^{-1}$ and $J$, and we claim that $\Pi_f$
%preserves $\kappa$-filtered colimits for any regular cardinal $\kappa > \abs
%X$. Indeed, for a $\kappa$-filtered diagram $A \colon \cat I \to \cat {Set}^X$,
%we have that:
%\begin{equation*}
%  \Pi_f(\colim_i A_i)(y) = \prod_{fx = y} \colim_i A_i(x)
%    = \colim_i \prod_{fx = y} A_i(x)
%    = \colim_i \Pi_f(A_i)(y)
%\end{equation*}
%since the displayed product is of cardinality $< \kappa$, and hence commutes
%with $\kappa$-filtered colimits in $\cat{Set}$.
%Thus $I$ gives rise to a
and take $(f, f^\ast) \colon (X, I) \to (Y, J)$ to be the map corresponding,
under Remark~\ref{alternatemap}, to the natural transformation $f^{-1} . J .
\eta \colon f^{-1} . J \Rightarrow f^{-1} . J . \Pi_f . f^{-1} = I . f^{-1}$.
It is easy to check that this map is cartesian; and so $U$ is a fibration.
%is cartesian. For this we must show that any map $(fg, h^\ast) \colon (Z, K)
%\to (Y, J)$ in $\cat{BIon}$ admits a unique factorisation $(fg, h^\ast) = (f,
%f^\ast) \circ (g, g^\ast)$. Now, to give $h^\ast$ and $g^\ast$ is equally well
%to give natural transformations $\gamma \colon g^{-1} . f^{-1} . J \Rightarrow
%K . g^{-1} . f^{-1}$ and $\beta \colon g^{-1} . I \Rightarrow K . g^{-1}$
%compatible with the comonad structures, and to ask that $h^\ast = f^\ast
%g^\ast$ is to ask that
%\begin{equation*}
%  \cd{
%    g^{-1} . f^{-1} . J \ar[rr]^{g^{-1} . \delta} \ar[dr]_{\gamma} & &
%    g^{-1} . I . f^{-1} \ar[dl]^{\beta . f^{-1}} \\ &
%    K . g^{-1} . f^{-1}
%  }
%\end{equation*}
%should commute. But since $g^{-1} . \delta = g^{-1} . f^{-1} . J . \eta$, we
%conclude by the triangle identities that $\beta$ must be the composite
%transformation
%\begin{equation*}
%  g^{-1} . I = g^{-1} . f^{-1} . J . \Pi_f \xrightarrow{\gamma . \Pi_f} K . g^{-1} . f^{-1} . \Pi_f \xrightarrow{K . g^{-1} . \epsilon} K . g^{-1}\ \text.
%\end{equation*}
%This forces uniqueness of the factorisation through $(f, f^\ast)$; and
%existence follows by checking that this $\beta$ respects the comonad structures
%of $I$ and $K$. This completes the verification that $(f, f^\ast)$ is

Secondly, we show that each of the fibres of $U$ is complete. For a given set
$X$, the fibre category $U_X$ is the opposite of the category of accessible,
cartesian comonads on $\cat{Set}^X$, and to show this complete, it suffices to
show that the category $\cat{AC}(\cat{Set}^X, \cat{Set}^X)$ of accessible,
cartesian endofunctors of $\cat{Set}^X$ is cocomplete. But this category is
isomorphic to $\cat{AC}(\cat{Set}^X, \cat{Set})^X$, so it's enough to show that
$\cat{AC}(\cat{Set}^X, \cat{Set})$ is cocomplete; which we do by proving it
reflective in the cocomplete $\cat{Acc}(\cat{Set}^X, \cat{Set})$. So given a
functor $A \colon \cat{Set}^X \to \cat{Set}$ which preserves $\kappa$-filtered
colimits, we let $\cat C$ denote a skeleton of the full subcategory of
$\cat{Set}^X$ on the $\kappa$-presentable objects; by elementary cardinal
arithmetic, $\cat C$ has finite limits and the inclusion $V \colon \cat C \to
\cat{Set}^X$ preserves them. The category $\cat{Cart}(\cat C, \cat{Set})$ is
reflective in $[\cat C, \cat{Set}]$; let $B$ denote the reflection of $A \circ
V$ into it. Now $\mathrm{Lan}_V B$ is clearly accessible, but we claim it is
also cartesian: whereupon it easily provides the required reflection of $A$
into $\cat{AC}(\cat{Set}^X, \cat{Set})$. To prove the claim, note that, since
$V$ is dense, $\Lan_V B$ is the composite of $B \otimes (\thg) \colon
\cat{Set}^{\cat C^\op} \to \cat{Set}$ with $[V, \thg] \colon \cat{Set}^X \to
\cat{Set}^{\cat C^\op}$. The former is cartesian because $B$ is, and the latter
because it is a right adjoint; so $\Lan_V B$ is cartesian as desired.

Thirdly, we show that for every map of sets $f \colon X \to Y$, the reindexing
functor $U_f \colon U_Y \to U_X$ preserves limits. This is equivalent to
showing that $U_f^\op$ is cocontinuous, for which it's enough to show that
\[ f^{-1} . (\thg) . \Pi_f \colon \cat{AC}(\cat{Set}^Y, \cat{Set}^Y) \to
   \cat{AC}(\cat{Set}^X, \cat{Set}^X)
\] is cocontinuous; but this is immediate from the fact that it has a right
adjoint $\Pi_f . (\thg) . f^{-1}$. This completes the proof that $\cat{BIon}$
is complete as a $1$-category.

To show that $\cat{BIon}$ is complete as a $2$-category, it now suffices to
show that it admits cotensors products with the arrow category $\mathbf 2$. If
$X$ is the ionad generated by a basis $M \colon \cat B \to \cat{Set}^X$, then
the cotensor product $\cat 2 \pitchfork X$ will be the ionad whose set of
points $Z$ is the set of triples $(x,y,\alpha)$, where $x, y \in X$ and $\alpha
\colon M(\thg)(x) \Rightarrow M(\thg)(y)$; observe that $Z$ is the set of
morphisms of the category $VX$ of Definition~\ref{specfunctor}. The topology on
this ionad is generated by the following basis $N \colon \cat B^{\cat 2} \to
\cat{Set}^{Z}$. For an object $k \colon c \to d$ of $\cat B ^{\cat 2}$ and
element $\alpha \colon M(\thg)(x) \Rightarrow M(\thg)(y)$ of $Z$, the set
$N(k)(\alpha)$ is obtained as the pullback
\begin{equation}\label{Ndef}
    \cd[@C+1em]{
      N(k)(\alpha) \ar[r] \ar[d] \pushoutcorner &
      M(d)(x) \ar[d]^{\alpha_{d}} \\
      M(c)(y) \ar[r]_{M(k)(y)} &
      M(d)(y)\rlap{ .}}
\end{equation}
With some effort, we may check that $N$ is flat; and with considerable further
effort, may verify the ionad it induces does indeed possess the universal
property required of the cotensor $\cat 2 \pitchfork X$. The argument is
similar to, but more elaborate than, the one given in the following Remark in
regard of finite products; it is closely related to the corresponding
topos-theoretic argument, as given in~\cite[Proposition
B4.1.2]{Johnstone2002Sketches}, for example.
\end{proof}

\begin{Rk}\label{finiteprod}
We may describe the product of two bounded ionads more concretely, in terms of
a basis generated by open rectangles. More precisely, if the ionads $X$ and $Y$
are generated by bases $M \colon \cat B \to \cat{Set}^X$ and $N \colon \cat C
\to \cat{Set}^Y$, then their product is the ionad with set of points $X \times
Y$ and topology generated by the basis
\begin{align*}
M \otimes N \colon \cat B \times \cat C
& \to \cat{Set}^{X \times Y}\\
    (M \otimes N)(b,c)(x,y) & = M(b)(x) \times N(c)(y)\ \text.
\end{align*}
Flatness of $M \otimes N$ follows easily from that of $M$ and $N$. To show that
the ionad it generates is a product of $X$ and $Y$, we must show that, for any
pair of functions $f \colon Z \to X$, $g \colon Z \to Y$, we have a bijection
between squares of the form
\begin{equation}\label{hdata}
    \cd[@C+1em]{
    \cat B \times \cat C \ar[r]^{h'} \ar[d]_{M \otimes N} &
    \cat O(Z) \ar[d]^{U_Z} \\
    \cat{Set}^{X \times Y} \ar[r]_-{(f,g)^{-1}} &
    \cat{Set}^Z
}
\end{equation}
and pairs of squares of the form
\begin{equation}\label{fgdata}
    \cd{
    \cat B \ar[r]^{f'} \ar[d]_{M} &
    \cat O(Z) \ar[d]^{U_Z} \\
    \cat{Set}^{X} \ar[r]_{f^{-1}} &
    \cat{Set}^Z
}\qquad \text{and} \qquad
    \cd{
    \cat C \ar[r]^{g'} \ar[d]_{N} &
    \cat O(Z) \ar[d]^{U_Z} \\
    \cat{Set}^{Y} \ar[r]_{g^{-1}} &
    \cat{Set}^Z\rlap{ .}
}
\end{equation}
%its category of elements is the product $\elwe may verify that this is indeed a
%basis, and that the ionad it generates has the universal property of the
%product $X \times Y$ in $\cat{BIon}$.
%Observe first that this is indeed a basis, since for each $(x, y)
%\in X \times Y$, the category of elements of $(M \otimes N)(\thg)(x,y)$ is
%isomorphic to $\mathrm{el}\ M(\thg)(x) \times \mathrm{el}\ N(\thg)(y)$
%which---being the product of two cofiltered categories---is cofiltered. It
%remains to show that the ionad generated by this basis is a product of $X$ and
%$Y$. For this, we must exhibit a natural bijection between ionad morphisms $Z
%\to X \times Y$ and pairs of ionad morphisms $Z \to X$ and $Z \to Y$. It is
%clear how to do this at the level of underlying sets; and to lift this to the
%level of continuous maps, it suffices by Example~\ref{basisex} to show that for
%any pair of functions $f \colon Z \to X$, $g \colon Z \to Y$, we have a natural
%bijection between squares of the form
%\begin{equation}\label{hdata}
%    \cd[@C+1em]{
%    \cat B \times \cat C \ar[r]^{h'} \ar[d]_{M \otimes N} &
%    \cat O(Z) \ar[d]^{U_Z} \\
%    \cat{Set}^{X \times Y} \ar[r]_-{(f,g)^{-1}} &
%    \cat{Set}^Z
%}
%\end{equation}
%and pairs of squares of the form
%\begin{equation}\label{fgdata}
%    \cd{
%    \cat B \ar[r]^{f'} \ar[d]_{M} &
%    \cat O(Z) \ar[d]^{U_Z} \\
%    \cat{Set}^{X} \ar[r]_{f^{-1}} &
%    \cat{Set}^Z
%}\qquad \text{and} \qquad
%    \cd{
%    \cat C \ar[r]^{g'} \ar[d]_{N} &
%    \cat O(Z) \ar[d]^{U_Z} \\
%    \cat{Set}^{Y} \ar[r]_{g^{-1}} &
%    \cat{Set}^Z\rlap{ .}
%}
%\end{equation}
On the one hand, if given $f'$ and $g'$ as in~\eqref{fgdata}, then we define
the corresponding $h'$  by $h'(b, c) = f'(b) \times g'(c)$; since $U_Z$
strictly creates finite limits, we may always choose this product in such a way
as to make~\eqref{hdata} commute. Conversely, if given $h'$ as
in~\eqref{hdata}, we define the corresponding $f'$ and $g'$ by $f'(b) = \int^{c
\in \mathbf C} h'(b,c)$ and $g'(c) = \int^{b \in \mathbf B} h'(b,c)$; again,
since $U_Z$ strictly creates colimits, we may choose the colimits in question
so as to render the squares in~\eqref{fgdata} commutative.
\end{Rk}
\begin{Rk}\label{alexspecionad}
We may also describe the tensor product of a bounded ionad by a small category
in terms of bases: given a basis $M \colon \cat B \to \cat{Set}^X$ for an ionad
$X$, easy calculation shows that $\cat C \otimes X$ may be generated by the
basis
\begin{align*}
N \colon \cat C^\op \times \cat B
& \to \cat{Set}^{\ob \cat C \times X}\\
    N(c,b)(c',x) & = \cat C(c',c) \times M(b)(x)\ \text.
\end{align*}
Now by comparing this description with Examples~\ref{ordgen}.1 and
Remark~\ref{finiteprod}, we conclude that for a bounded ionad $X$, the tensor
product $\cat C \otimes X$ is equally well the product $A(\cat C) \times X$; in
particular, there is a $2$-natural isomorphism $A \cong (\thg) \otimes 1 \colon
\cat{Cat} \to \cat{BIon}$, and so by virtue of the $2$-adjunction
\begin{equation*}
    (\thg) \otimes 1 \dashv \cat{BIon}(1, \thg) \colon \cat{BIon} \to \cat{Cat}
\end{equation*}
we deduce, as promised in Remark~\ref{alexspec}, that the Alexandroff embedding
$A \colon \cat{Cat} \to \cat{BIon}$ is left $2$-adjoint to the specialisation
$2$-functor $V \colon \cat{BIon} \to \cat{Cat}$.
\end{Rk}

\section{Conclusions}\label{s6}
In this final section, we make a few comments on the advantages and
disadvantages of the notion of ionad as compared with the notion of topos. The
obvious starting point for such a discussion is a consideration of the
analogous relationship between the notions of topological space and locale.

One of the major advantages that locales have over spaces is the ease with
which their theory may be \emph{relativised}. Maps of locales $X \to Y$ may be
identified with internal locales in the sheaf topos $\cat{Sh}(Y)$, and so
properties of, and constructions on, locales---so long as these are expressed
in the logic common to any topos---may without effort be transferred to
properties of, and constructions on, maps of locales. For instance, as soon as
we know how to form the product of locales, we also know how to form the fibre
product over $X$: it is simply the product of locales internal to
$\cat{Sh}(X)$. The theory of topological spaces does not relativise in the same
way, since many parts of its development makes essential use of classical
logic, and so do not internalise well to an arbitrary topos.

It seems likely that this advantage of locales over spaces propagates upwards
to a corresponding advantage of toposes over ionads. Certainly, the theory of
toposes relativises very satisfactorily: for example, bounded geometric
morphisms into a topos $\F$ correspond with internal sites in that topos; and
this means that, for example, constructing the pullback of bounded geometric
morphisms is scarely more problematic than constructing the product of two
Grothendieck toposes. Yet it seems unlikely that the theory of ionads
relativises in the same manner: so, for instance, we should not expect our
concrete description of the product of two bounded ionads to yield a
corresponding concrete description of pullbacks of bounded ionads. This, then,
is one reason for preferring toposes over ionads.

A second reason is that many toposes of interest do not have a natural
expression as an ionad. Most obviously, this could be because the topos we are
interested in does not have enough points: which mirrors the corresponding fact
that a non-spatial locale will not admit a natural expression as a topological
space. More subtly, it could be that the topos we are interested in has
\emph{too many} points: namely, a proper class of them. Such a topos, if
spatial, will admit any number of different representations as an ionad, but
each such representation will require the selection of a mere \emph{set} of
separating points: and since maps of ionads are required to preserve these
selected sets of points, none of the ionads representing the topos will be able
to capture the full range of geometric morphisms into it. This means, amongst
other things, that the theory of classifying toposes has no ionad-theoretic
analogue. For instance, there can be no bounded ionad~$Y$ which ``classifies
groups'' in the sense that ionad morphisms $X \to Y$ correspond with group
objects in $\cat O(X)$. The best we can do is to construct, as in
Examples~\ref{ordgen}.5, the ionad $Y$ which ``classifies $\lambda$-small
groups''---in the sense that ionad morphisms $X \to Y$ correspond with group
objects $G \in \cat O(X)$ whose ``stalks are $\lambda$-small''; in other words,
such that $(U_X G)(x)$ is a set of cardinality $< \lambda$ for each $x \in X$.
Clearly this is nowhere near as useful a notion, which is something of a pity:
the classifying topos of groups really should be considered as ``the
generalised space of all groups equipped with the Scott topology'', and the
language of ionads would appear ideal for the expression of this idea. It is
conceivable that this problem could be overcome with a sufficiently clever
definition of ``large ionad''---one endowed with a proper class of points---but
whilst there are a few obvious candidates for such a notion, none seems to be
wholly satisfactory.

These, then, are two quite general grounds for preferring toposes over ionads;
yet there remain good reasons for having the notion of ionad available to us.
The first is that some particular applications of topos theory may be more
perspicuously expressed in the language of ionads than of toposes: two examples
that come to mind are the sheaf-theoretic semantics for first-order modal logic
given in~\cite{AK}, and the generalised Stone duality of~\cite{F}. The second
reason is pedagogical. Many aspects of topos theory are abstractions of
corresponding aspects of general topology, but the abstraction is twice
removed: first one must pass from spaces to locales, and then from locales to
toposes. At the first step, one loses the points, which to many, is already to
enter a quite unfamiliar world, and the second step can only compound this
unfamiliarity. With the notion of ionad available, one may arrive at these same
abstractions by a different route, passing first from spaces to ionads, and
then from ionads to toposes. The advantage of doing so is that one retains the
tangibility afforded by the presence of points for as long as possible. It
seems to me that it is in this pedagogical aspect that ionads are likely to
make their most useful contribution.

%%%% Bibliography

%\bibliographystyle{abbrv}
%\bibliography{rhgg2}

\end{document}